\documentclass{article}

\usepackage{a4}
\usepackage{paralist}
\usepackage{graphicx, psfrag}
\usepackage{pgfplotstable}
\usepackage{subcaption}

\title{An adaptive finite element method for distributed elliptic optimal
control problems with \\ variable energy regularization}
\author{Ulrich~Langer\footnote{Institute of Computational Mathematics,
    Johannes Kepler University Linz, and RICAM, \"OAW, Altenberger Stra{\ss}e 69, 4040 Linz,
    Austria, Email: ulanger@numa.uni-linz.ac.at},
  \;
  Richard~L\"oscher\footnote{Institut f\"{u}r Angewandte Mathematik,
    Technische Universit\"{a}t Graz, Steyrergasse 30, 8010 Graz, Austria,
    Email: loescher@math.tugraz.at}, 
  \; Olaf~Steinbach\footnote{Institut f\"{u}r Angewandte Mathematik,
    Technische Universit\"{a}t Graz, Steyrergasse 30, 8010 Graz, Austria,
    Email: o.steinbach@tugraz.at}, 
  \; Huidong~Yang\footnote{Faculty of Mathematics, University of Vienna,
  Oskar--Morgenstern--Platz 1, 1090 Wien, Austria, and Christian
  Doppler Laboratory for Mathematical Modeling and Simulation of Next
  Generations of Ultrasound Devices (MaMSi),
  Oskar--Morgenstern--Platz 1, 1090 Wien, Austria, Email: huidong.yang@univie.ac.at}
}  

\date{}

\usepackage{amsmath}
\usepackage{amsthm}
\usepackage{amsfonts}
\usepackage{booktabs}
\usepackage{graphicx}
\usepackage{diagbox}

\usepackage{multirow}
\usepackage{hyperref}

\usepackage[ulem=normalem,draft]{changes}

\newtheorem{theorem}{Theorem}
\newtheorem{lemma}{Lemma}

\numberwithin{equation}{section} 

\begin{document}

\maketitle

\begin{abstract}
We analyze the finite element discretization of distributed elliptic 
optimal control problems with variable energy regularization, where the
usual $L^2(\Omega)$ norm regularization term with a constant regularization
parameter $\varrho$ is replaced by a suitable representation of the
energy norm in $H^{-1}(\Omega)$
involving a variable, mesh-dependent regularization
parameter $\varrho(x)$. It turns out that the error between the
computed finite element state $\widetilde{u}_{\varrho h}$ and the desired
state $\overline{u}$ (target) is optimal in the $L^2(\Omega)$ norm 
provided that $\varrho(x)$ behaves like the local mesh size squared.
This is especially important when adaptive meshes are used in order
to approximate discontinuous target functions.
The adaptive scheme can be driven by the computable and localizable 
error norm $\| \widetilde{u}_{\varrho h} - \overline{u}\|_{L^2(\Omega)}$ 
between the finite element state $\widetilde{u}_{\varrho h}$ and the target
$\overline{u}$. The numerical results not only illustrate our theoretical
findings, but also show that the iterative solvers for the discretized reduced 
optimality system are very efficient and robust.
\end{abstract} 

\noindent
\begin{msc}
49J20,  
49M05,  
35J05,  
65M60,  
65M15,  
65N22   
\end{msc}

%
%
\section{Introduction}
\label{sec:Introduction}
%
Optimal control \cite{ABVS2012,MHRPMUSU2009,Lions:1971,Troeltzsch:2010}
and inverse problems
\cite{EnglHankeNeubauer:1996,Isakov:2017,SchusterKaltenbacher:2012}
subject to partial differential equations often involve some
parameter-dependent cost or regularization terms, see also the recent
special issue \cite{ClasonKaltenbacher:2020} on optimal control and
inverse problems. While in optimal control problems the regularization
parameter $\varrho$ is often considered as a given constant, in inverse
problems, the parameter-dependent convergence as $\varrho \to 0$
is well studied, see, e.g., \cite{AlbaniCezaroZubelli:2016}.
For tracking type cost functionals subject to elliptic partial differential
equations, the regularization error was analyzed in
\cite{LLSY2:NeumuellerSteinbach:2021a} which depends on the regularity
of the given target. In \cite{LLSY2:LangerSteinbachYang:2022b},
and in the case of energy regularization, we have
considered a related finite element analysis which resulted in an
optimal choice of the regularization parameter $\varrho = h^2$,
or vice versa, where
$h$ is the mesh size of the globally quasi-uniform finite element mesh.
While the latter can be used to approximate smooth target functions
in the state space, adaptively refined finite element meshes should be
used when considering less regular target functions, e.g., discontinuous
or singular, or violating Dirichlet boundary conditions which are
involved in the definition of the state space. This motivates the use
of a variable regularization parameter function which 
can be defined by using the local finite element mesh sizes $h_\ell$. 
We are not aware of any paper on optimal control problems dealing with such
an approach. However, there are several papers in imaging considering
a similar approach: In \cite{Burger:2017}, a variable regularization
parameter is used for an adaptive balancing of the data fidelity and
the regularization term, see equation (3) in \cite{Burger:2017}.
A variable $L^2(\Omega)$ (TV) regularization is considered in
\cite[equation (1.1)]{Chung:2017}. Finally, a spatially adapted total
generalized variation model was already
used in \cite[equation (1.4)]{Bredies:2013}, see also the more recent
work \cite{Hintermueller:2022}.

As model problem we consider the optimal control problem to minimize 
the cost functional
\begin{equation}\label{minimization problem}
  {\mathcal{J}}(u_\varrho,z_\varrho) =
  \frac{1}{2} \int_\Omega [u_\varrho(x)-\overline{u}(x)]^2 dx +
  \frac{1}{2} \, \varrho \, \| z_\varrho \|^2_{H^{-1}(\Omega)}
\end{equation}
subject to the Dirichlet boundary value problem for the Poisson equation,
\begin{equation}\label{primal problem}
  - \Delta u_\varrho = z_\varrho \quad \mbox{in} \; \Omega, \quad
  u_\varrho = 0 \quad \mbox{on} \; \partial\Omega .
\end{equation}
We assume that $\Omega \subset {\mathbb{R}}^n$, $n=1,2,3$, is a bounded
Lipschitz domain, and $\varrho \in {\mathbb{R}}_+$ is some, at this time
constant, regularization parameter, on which the solution
$(u_\varrho,z_\varrho)$ depends.
Our particular interest is in the behavior of
$\| u_\varrho - \overline{u}\|_{L^2(\Omega)}$ as $\varrho \to 0$, see
\cite{LLSY2:NeumuellerSteinbach:2021a}. Note that the energy norm as used in
\eqref{minimization problem} is given by
\[
  \| z_\varrho \|_{H^{-1}(\Omega)}^2 = \| \nabla u_\varrho \|^2_{L^2(\Omega)} =
  \langle z_\varrho , u_\varrho \rangle_\Omega, =
  \langle z_\varrho , {\mathcal{S}} z_\varrho \rangle_\Omega,
\]
where $u_\varrho = {\mathcal{S}} z_\varrho \in H^1_0(\Omega)$ is the weak
solution of the primal Dirichlet boundary value problem \eqref{primal problem},
and ${\mathcal{S}} : H^{-1}(\Omega) \to H^1_0(\Omega) \subset L^2(\Omega)$
is the associated solution operator. Hence we can write the reduced
cost functional as
\begin{equation}\label{reduced cost rho konstant}
  \widetilde{\mathcal{J}}(z_\varrho) =
  \frac{1}{2} \, \| {\mathcal{S}} z_\varrho - \overline{u} \|^2_{L^2(\Omega)}
  +
  \frac{1}{2} \, \varrho \, \langle {\mathcal{S}} z_\varrho , z_\varrho
  \rangle_\Omega,
\end{equation}
and its minimizer is given as the unique solution of the
gradient equation
\[
  {\mathcal{S}}^* ({\mathcal{S}}z_\varrho - \overline{u}) +
  \varrho \, {\mathcal{S}}z_\varrho = 0 \, .
\]
In addition to $u_\varrho = {\mathcal{S}}z_\varrho$ we now introduce the
adjoint state $p_\varrho = {\mathcal{S}}^*(u_\varrho-\overline{u})$
as unique solution of the Dirichlet boundary value problem
\begin{equation}\label{adjoint problem}
  - \Delta p_\varrho = u_\varrho - \overline{u} \quad
  \mbox{in} \; \Omega, \quad p_\varrho =0 \quad
  \mbox{on} \; \partial \Omega .
\end{equation}
Hence we can rewrite the gradient equation as
\[
p_\varrho + \varrho \, u_\varrho = 0 \quad \mbox{in} \; \Omega .
\]
When eliminating the adjoint state $p_\varrho$ we can determine the
optimal state $u_\varrho$ as the solution of the singularly perturbed
Dirichlet boundary value problem
\begin{equation}\label{Gradient equation rho konstant}
  - \varrho \, \Delta u_\varrho + u_\varrho = \overline{u} \quad
  \mbox{in} \; \Omega, \quad u_\varrho = 0 \quad \mbox{on} \;
  \partial \Omega,
\end{equation}
which is also known as differential filter in fluid mechanics \cite{LLSY2:John:2016a}.
The variational formulation of \eqref{Gradient equation rho konstant}
is to find $u_\varrho \in H^1_0(\Omega)$ such that
\begin{equation}\label{VF rho konstant}
  \varrho \, \langle \nabla u_\varrho , \nabla v \rangle_{L^2(\Omega)} +
  \langle u_\varrho , v \rangle_{L^2(\Omega)} =
  \langle \overline{u} , v \rangle_{L^2(\Omega)} \quad
  \mbox{for all} \; v \in H^1_0(\Omega) .
\end{equation}
For a finite element discretization of the variational formulation
\eqref{Gradient equation rho konstant} we may use the ansatz space
$V_h := S_h^1(\Omega) \cap H^1_0(\Omega)$
of piecewise linear and continuous basis functions which are defined with
respect to some admissible and globally quasi-uniform decomposition
of $\Omega$ into simplicial finite elements of mesh size $h$.
The Galerkin variational formulation of \eqref{VF rho konstant}
is to find $u_{\varrho h} \in V_h$ such that
\begin{equation}\label{FEM rho konstant}
  \varrho \, \langle \nabla u_{\varrho h} , \nabla v_h \rangle_{L^2(\Omega)} +
  \langle u_{\varrho h}, v_h \rangle_{L^2(\Omega)} =
  \langle \overline{u} , v_h \rangle_{L^2(\Omega)} \quad
  \mbox{for all} \; v_h \in V_h.
\end{equation}
When combining the regularization error estimates for
$\| u_\varrho - \overline{u} \|_{L^2(\Omega)}$ as given in
\cite{LLSY2:NeumuellerSteinbach:2021a} with finite element error estimates for the
approximate solution $u_{\varrho h}$, i.e., for
$\| u_{\varrho h} - u_\varrho \|_{L^2(\Omega)}$,
we were able to derive estimates
for the error $\| u_{\varrho h} - \overline{u} \|_{L^2(\Omega)}$, see
\cite{LLSY2:LangerSteinbachYang:2022b}. In particular for the optimal choice
$\varrho = h^2$ this gives
\begin{equation}\label{Error rho konstant}
  \| u_{\varrho h} - \overline{u} \|_{L^2(\Omega)} \leq c \, h^s \,
  \| \overline{u} \|_{H^s(\Omega)},
\end{equation}
when assuming $\overline{u} \in H^s_0(\Omega) :=
[L^2(\Omega),H^1_0(\Omega)]_s$ for $s \in [0,1]$, or
$ \overline{u} \in H^1_0(\Omega) \cap H^s(\Omega)$ for $ s \in (1,2]$.
This error estimate remains true when considering optimal
control problems with energy regularization subject to time-dependent
partial differential equations, see
\cite{UL:LangerSteinbachYang:2022a} in the case of
the heat equation. However, when considering less regular target functions
$\overline{u}$, e.g., singular or discontinuous targets, the use of
adaptive finite elements seems to be mandatory in order to gain optimal
computational complexity. In this case it is not obvious how to choose
a constant regularization parameter $\varrho$, e.g.,
$\varrho = h_{\min{}}^2$. Instead, one may use a locally adapted
regularization function $\varrho(x)$, $x \in \Omega$.
%
The energy norm in $H^{-1}(\Omega)$ can be realized by duality when
  solving a Poisson equation with zero Dirichlet boundary conditions.
  When including the (constant) regularization parameter $\varrho$,
  we can generalize this approach by considering a diffusion equation
  with $\varrho(x)^{-1}$ as diffusion coefficient in order to realize
  the variable energy regularization within an adaptive finite element
  discretization.

The rest of this paper is organized as follows. In
Section~\ref{sec:DistributedOCPs}, we derive the analogon of the 
optimal control problem \eqref{minimization problem} when using  
a regularization function $\varrho(x)$ instead of a constant 
regularization parmeter $\varrho$, and the corresponding 
reduced optimality systems. 
Section~\ref{sec:RegularizationErrorEstimates} provides estimates
of the derivation of the state $u_\varrho$ from the desired state
$\overline{u}$ with respect to the $L^2$-norm in terms of the
regularization function $\varrho(x)$, and the regularity of the
target $\overline{u}$.
In Section~\ref{sec:FiniteElementErrorEstimates}, we analyze the
$L^2$-norm $\|\widetilde{u}_{\varrho h} - \overline{u}\|_{L^2(\Omega)}$
of the error between the computed finite element state
$\widetilde{u}_{\varrho h}$ and the desired state $\overline{u}$ leading
to an elementwise adaption of the regularization function $\varrho(x)$
to the local mesh size $h_\ell$.
The first part of Section~\ref{sec:Numerical results} is devoted 
to numerical studies of the error
$\|\widetilde{u}_{\varrho h} - \overline{u}\|_{L^2(\Omega)}$
for benchmark problems with discontinuous targets $\overline{u}$, 
the second part 
devises 
a postprocessing algorithm for the 
computation of the control corresponding to the computed state,
whereas the third
part provides numerical studies of the proposed iterative 
solvers in the three-dimensional case $n=3$.
Finally, in Section~\ref{sec:ConclusionAndOutlook},
we draw some conclusions, and give some outlook on further research work.

%
%
\section{Distributed optimal control problems with \\
  variable energy regularization}
\label{sec:DistributedOCPs}
To give a motivation for the optimization problem to be solved, let
us consider an alternative representation of the energy norm, still
using a constant regularization parameter $\varrho$:
\[
  \varrho \, \| z_\varrho \|^2_{H^{-1}(\Omega)} =
  \| \sqrt{\varrho} \, z_\varrho \|^2_{H^{-1}(\Omega)} =
  \langle \sqrt{\varrho} \, z_\varrho , w_\varrho \rangle_\Omega,
\]
where $w_\varrho \in H^1_0(\Omega)$ is the weak solution of the Dirichlet
boundary value problem
\begin{equation}\label{norm splitted}
  - \Delta w_\varrho = \sqrt{\varrho} \, z_\varrho \quad
  \mbox{in} \; \Omega, \quad w_\varrho = 0
  \quad \mbox{on} \; \partial \Omega .
\end{equation}
Now we may introduce
$\widetilde{w}_\varrho = \sqrt{\varrho} \, w_\varrho$
to conclude the diffusion equation
\[
  - \mbox{div} \left[ \varrho^{-1} \, \nabla \widetilde{w}_\varrho \right] =
  z_\varrho \quad \mbox{in} \; \Omega, \quad
  \widetilde{w}_\varrho = 0 \quad \mbox{on} \; \partial \Omega ,
\]
and the norm representation
\[
  \varrho \, \| z_\varrho \|^2_{H^{-1}(\Omega)} =
  \langle \sqrt{\varrho} \, z_\varrho , w_\varrho \rangle_\Omega =
  \langle z_\varrho , \widetilde{w}_\varrho \rangle_\Omega .
\]
Instead of using a constant regularization parameter $\varrho$ we now
consider a diffusion equation with a variable diffusion coefficient 
$\varrho\in L^\infty(\Omega)$, that is uniformely bounded from above and below, 
i.e., there exists positive constants $\underline \varrho$ and $\overline \varrho$ such that 
$0<\underline \varrho \leq \varrho(x)\leq \overline \varrho <\infty$
for almost 
all
$x\in\Omega$.
More precisely,
we consider the Dirichlet boundary value problem
\begin{equation}\label{norm diffusion}
  - \mbox{div} \left[ \frac{1}{\varrho(x)} \,
    \nabla \widetilde{w}_\varrho(x) \right] =
  z_\varrho(x) \quad \mbox{for} \; x \in \Omega, \quad
  \widetilde{w}_\varrho(x) = 0 \quad \mbox{for} \; x \in \partial \Omega ,
\end{equation}
and its variational formulation to find
$\widetilde{w}_\varrho \in H^1_0(\Omega)$
such that
\[
  \langle A_{1/\varrho} \widetilde{w}_\varrho , v \rangle_\Omega :=
  \int_\Omega \frac{1}{\varrho(x)} \, \nabla \widetilde{w}_\varrho(x) \cdot
  \nabla v(x) \, dx = \int_\Omega z_\varrho(x) \, v(x) \, dx
\]
is satisfied for all $v \in H^1_0(\Omega)$, i.e., we have
$\widetilde{w}_\varrho = A^{-1}_{1/\varrho} z_\varrho$. Instead of
\eqref{reduced cost rho konstant} we now consider the reduced
cost functional
\begin{equation}\label{cost functional diffusion}
  \widetilde{\mathcal{J}}(z_\varrho) =
  \frac{1}{2} \, \| {\mathcal{S}} z_\varrho - \overline{u} \|^2_{L^2(\Omega)}
  + \frac{1}{2} \, \langle A^{-1}_{1/\varrho} z_\varrho , z_\varrho
  \rangle_\Omega,
\end{equation}
whose minimizer is given as the unique solution of the gradient
equation
\[
  {\mathcal{S}}^* ({\mathcal{S}}z_\varrho - \overline{u}) +
  A^{-1}_{1/\varrho} z_\varrho = 0,
\]
i.e.,
\begin{equation}\label{gradient equation diffusion}
  p_\varrho + \widetilde{w}_\varrho = 0 \quad \mbox{in} \; \Omega .
\end{equation}
The optimality system to be solved is now given by the primal
problem \eqref{primal problem}, the adjoint problem \eqref{adjoint problem},
the gradient equation \eqref{gradient equation diffusion},
and \eqref{norm diffusion}.
When eliminating $\widetilde{w}_\varrho$ and $z_\varrho$,
we end up with a coupled system of the primal problem
\begin{equation}
  - \mbox{div} \left[ \frac{1}{\varrho(x)} \, \nabla p_\varrho(x) \right]
  - \Delta u_\varrho(x) = 0 \quad \mbox{for} \; x \in \Omega, \quad
  u_\varrho(x) = 0 \quad \mbox{for} \; x \in \partial \Omega ,
\end{equation}
and the adjoint boundary value problem \eqref{adjoint problem}.
The related variational
formulation is to find $(u_\varrho,p_\varrho) \in H^1_0(\Omega) \times
H^1_0(\Omega)$ such that
\begin{equation}\label{VF primal}
  \int_\Omega \frac{1}{\varrho(x)} \, \nabla p_\varrho(x) \cdot
  \nabla v(x) \, dx + \int_\Omega \nabla u_\varrho(x) \cdot \nabla v(x)
  \, dx = 0
\end{equation}
for all $v \in H^1_0(\Omega)$, and
\begin{equation}\label{VF adjoint}
  \int_\Omega u_\varrho(x) \, q(x) \, dx
  - \int_\Omega \nabla p_\varrho(x) \cdot \nabla q(x) \, dx
  =
  \int_\Omega \overline{u}(x) \, q(x) \, dx
\end{equation}
for all $q \in H^1_0(\Omega)$. When introducing
\[
  \langle B u , v  \rangle_\Omega :=
  \int_\Omega \nabla u(x) \cdot \nabla v(x) \, dx \quad
  \mbox{for all} \; u,v \in H^1_0(\Omega),
\]
we can write the coupled variational formulation \eqref{VF primal}
and \eqref{VF adjoint} in operator form as
\[
  A_{1/\varrho} p_\varrho + B u_\varrho = 0, \quad
  u_\varrho - B^* p_\varrho = \overline{u},
\]
and eliminating $p_\varrho$ results in the Schur complement system
to find $u_\varrho \in H^1_0(\Omega)$ such that
\begin{equation}\label{Operator diffusion}
B^* A^{-1}_{1/\varrho} B u_\varrho + u_\varrho = \overline{u} .
\end{equation}
Note that
\begin{equation}\label{Def S diffusion}
S_\varrho := B^* A^{-1}_{1/\varrho} B : H^1_0(\Omega) \to H^{-1}(\Omega)
\end{equation}
is bounded, self-adjoint, and $H^1_0(\Omega)$ elliptic. Note that
for a constant regularization parameter $\varrho(x)=\varrho$,
\eqref{Operator diffusion} coincides with
\eqref{Gradient equation rho konstant}.

%
%
\section{Regularization error estimates}
\label{sec:RegularizationErrorEstimates}

In this section, we consider estimates for the regularization error
$\| u_\varrho - \overline{u} \|_{L^2(\Omega)}$ when
$u_\varrho \in H^1_0(\Omega)$ is the weak solution of the operator
equation
\begin{equation}\label{Abstract operator equation}
S_\varrho u_\varrho + u_\varrho = \overline{u} \, ,
\end{equation}
and where $S_\varrho$ is as defined in \eqref{Def S diffusion}. Note that
$S_\varrho$ induces a norm, satisfying
\[
  \| v \|_{S_\varrho} := \sqrt{\langle S_\varrho v , v \rangle_\Omega},
  \quad
  \langle S_\varrho u , v \rangle_\Omega \leq
  \| u \|_{S_\varrho} \| v \|_{S_\varrho} \quad
  \mbox{for all} \; u,v \in H^1_0(\Omega).
\]
First we follow the general approach as given in \cite[Section~2]{UL:LangerSteinbachYang:2022a}
in the case of a constant regularization parameter. The variational
formulation of the
operator equation \eqref{Abstract operator equation} is to find
$u_\varrho \in H^1_0(\Omega)$ such that
\begin{equation}\label{Abstract VF}
  \langle S_\varrho u_\varrho , v \rangle_\Omega +
  \langle u_\varrho , v \rangle_{L^2(\Omega)} =
  \langle \overline{u} , v \rangle_{L^2(\Omega)} 
\end{equation}
is satisfied for all $v \in H^1_0(\Omega)$. Unique solvability of the
variational formulation \eqref{Abstract VF} follows from the boundedness
and ellipticity of $S_\varrho$.

\begin{theorem}
  Let $u_\varrho \in H^1_0(\Omega)$ be the unique solution of the
  variational formulation \eqref{Abstract VF}.
  For $ \overline{u} \in L^2(\Omega)$, there holds the estimate
  \begin{equation}\label{abstract regularization L2 L2}
    \| u_\varrho - \overline{u} \|_{L^2(\Omega)} \leq
    \| \overline{u} \|_{L^2(\Omega)} .
  \end{equation}
  For $\overline{u} \in H^1_0(\Omega)$ we have
  \begin{equation}\label{abstract regularization H1 H1}
    \| u_\varrho - \overline{u} \|_{S_\varrho} \leq
    \|\overline{u} \|_{S_\varrho},
  \end{equation}
  and
  \begin{equation}\label{abstract regularization L2 H1}
    \| u_\varrho - \overline{u} \|_{L^2(\Omega)} \leq
    \| \overline{u} \|_{S_\varrho} \, .
  \end{equation}
  If $\overline{u} \in H^1_0(\Omega)$ is such that
  $S_\varrho \overline{u} \in L^2(\Omega)$ is satisfied, we also
  have
  \begin{equation}\label{abstract regularization L2 H2}
    \| u_\varrho - \overline{u} \|_{L^2(\Omega)} \leq
    \| S_\varrho \overline{u} \|_{L^2(\Omega)},
  \end{equation}
  and
  \begin{equation}\label{abstract regularization H1 H2}
    \| u_\varrho -\overline{u} \|_{S_\varrho} \leq
    \| S_\varrho \overline{u} \|_{L^2(\Omega)} .
  \end{equation}
\end{theorem}

\begin{proof}
  When considering the variational formulation \eqref{Abstract VF} for
  $v=u_\varrho \in H^1_0(\Omega)$, this gives
  \[
    \langle S_\varrho u_\varrho , u_\varrho \rangle_\Omega +
    \langle u_\varrho , u_\varrho \rangle_{L^2(\Omega)} =
    \langle \overline{u} , u_\varrho \rangle_{L^2(\Omega)},
  \]
  which can be written as
  \[
    \langle S_\varrho u_\varrho , u_\varrho \rangle_\Omega +
    \langle u_\varrho - \overline{u} ,
    u_\varrho - \overline{u} \rangle_{L^2(\Omega)}
    =
    \langle \overline{u} - u_\varrho , \overline{u} \rangle_{L^2(\Omega)} ,
  \]
  i.e.,
  \[
    \langle S_\varrho u_\varrho , u_\varrho \rangle_\Omega +
    \| u_\varrho - \overline{u} \|_{L^2(\Omega)}^2 =
    \langle \overline{u} - u_\varrho , u_\varrho \rangle_{L^2(\Omega)}
    \leq \| u_\varrho - \overline{u} \|_{L^2(\Omega)}
    \| \overline{u} \|_{L^2(\Omega)} .
  \]
  Hence we conclude \eqref{abstract regularization L2 L2}.
  For $\overline{u} \in H^1_0(\Omega)$ we can consider the variational
  formulation \eqref{Abstract VF} for
  $v=\overline{u} - u_\varrho \in H^1_0(\Omega)$ to obtain
  \begin{eqnarray*}
    \| \overline{u} - u_\varrho \|^2_{L^2(\Omega)}
    & = &
    \langle \overline{u} - u_\varrho , \overline{u} - u_\varrho
    \rangle_{L^2(\Omega)} =
          \langle S_\varrho u_\varrho , \overline{u} - u_\varrho \rangle_\Omega \\
    & = & - \langle S_\varrho \overline{u} - u_\varrho ,
          \overline{u} - u_\varrho \rangle_\Omega + \langle S_\varrho
          \overline{u} , \overline{u} - u_\varrho \rangle_\Omega ,
  \end{eqnarray*}
  i.e.,
  \[
    \| u_\varrho - \overline{u} \|^2_{L^2(\Omega)} +
    \| u_\varrho - \overline{u} \|^2_{S_\varrho} =
    \langle S_\varrho \overline{u} , \overline{u} - u_\varrho \rangle_\Omega
    \leq \| \overline{u} \|_{S_\varrho}
    \| u_\varrho - \overline{u} \|_{S_\varrho} .
  \]
  From this we conclude
  \[
    \| u_\varrho - \overline{u} \|_{S_\varrho} \leq
    \| \overline{u} \|_{S_\varrho},
  \]
  that is \eqref{abstract regularization H1 H1}, and
  \[
    \| u_\varrho - \overline{u} \|_{L^2(\Omega)} \leq
    \| \overline{u} \|_{S_\varrho},
  \]
  i.e., \eqref{abstract regularization L2 H1}. If
  $\overline{u} \in H^1_0(\Omega)$ is such that $S_\varrho \overline{u} \in
  L^2(\Omega)$ is satisfied, we also have
  \[
    \| u_\varrho - \overline{u} \|^2_{L^2(\Omega)} +
    \| u_\varrho - \overline{u} \|^2_{S_\varrho} =
    \langle S_\varrho \overline{u} , \overline{u} - u_\varrho \rangle_\Omega
    \leq \| S_\varrho \overline{u} \|_{L^2(\Omega)}
    \| u_\varrho - \overline{u} \|_{L^2(\Omega)} .
  \]
  Hence we obtain
  \[
    \| u_\varrho - \overline{u} \|_{L^2(\Omega)} \leq
    \| S_\varrho \overline{u} \|_{L^2(\Omega)},
  \]
  that is \eqref{abstract regularization L2 H2}, and
  \[
    \| u_\varrho - \overline{u} \|_{S_\varrho} \leq
    \| S_\varrho \overline{u} \|_{L^2(\Omega)},
  \]
  i.e., \eqref{abstract regularization H1 H2}.
\end{proof}

Let us now consider the operator $S_\varrho$ as defined in
\eqref{Def S diffusion}. For $ u \in H^1_0(\Omega)$, let
$p_u = A^{-1}_{1/\varrho} Bu \in H^1_0(\Omega)$ be the unique solution
of the variational formulation
\[
  \langle A_{1/\varrho} p_u , v \rangle_\Omega =
  \int_\Omega \frac{1}{\varrho(x)} \, \nabla p_u(x) \cdot \nabla v(x) \, dx =
  \int_\Omega \nabla u(x) \cdot \nabla v(x) \, dx =
  \langle B u , v \rangle_\Omega
\]
for all $v \in H^1_0(\Omega)$. We first conclude
\[
  \| u \|_{S_\varrho}^2
  \, = \,
  \langle S_\varrho u , u \rangle_\Omega
  \, = \,
  \langle B^* A^{-1}_{1/\varrho} B u , u \rangle_\Omega
  \, = \,
  \langle p_u , B u \rangle_\Omega
  \, = \,
  \langle A_{1/\varrho} p_u , p_u \rangle_\Omega .
\]
Moreover, using a weighted Cauchy--Schwarz inequality, this gives
\begin{eqnarray*}
  \langle A_{1/\varrho} p_u , p_u \rangle_\Omega
  & = & \int_\Omega \frac{1}{\varrho(x)} \,
        \nabla p_u(x) \cdot \nabla p_u(x) \, dx \, = \,
        \int_\Omega \nabla u(x) \cdot \nabla p_u(x) \, dx \\
  & & \hspace*{-2cm}
      \leq \left( \int_\Omega \varrho(x) \, \nabla u(x) \cdot \nabla u(x) \, dx
           \right)^{1/2} \left(
           \int_\Omega \frac{1}{\varrho(x)} \,
           \nabla p_u(x) \cdot \nabla p_u(x) \, dx \right)^{1/2},
\end{eqnarray*}
i.e.,
\begin{equation}\label{Norm H1 diffusion}
  \| u \|_{S_\varrho}^2 \leq \int_\Omega \varrho(x) \, |\nabla u(x)|^2 \, dx
  \quad \mbox{for all} \; u \in H^1_0(\Omega) .
\end{equation}
When combining this with the regularization error estimate
\eqref{abstract regularization L2 H1} this gives, for
$\overline{u} \in H^1_0(\Omega)$,
\begin{equation}\label{regularization diffusion L2 H1}
  \| u_\varrho - \overline{u} \|_{L^2(\Omega)}^2 \leq
  \int_\Omega \varrho(x) \, |\nabla \overline{u}(x)|^2 \, dx \, .
\end{equation}
It remains to consider $\| S_\varrho u \|_{L^2(\Omega)}$ when
$S_\varrho$ is given as in \eqref{Def S diffusion}, i.e.,
\[
  \| S_\varrho u \|_{L^2(\Omega)} = \| - \Delta p_u \|_{L^2(\Omega)}, \quad
  - \mbox{div} \left[ \frac{1}{\varrho(x)} \nabla p_u(x) \right] =
  - \Delta u(x) .
\]
We first compute
\begin{eqnarray*}
  \varrho(x) \Delta u(x)
  & = & \varrho(x) \, \mbox{div} \left[
        \frac{1}{\varrho(x)} \nabla p_u(x) \right] \, = \,
        \varrho(x) \sum\limits_{k=1}^n \frac{\partial}{\partial x_k}
        \left[
        \frac{1}{\varrho(x)} \frac{\partial}{\partial x_k} p_u(x) \right]  \\
  & = & \varrho(x) \, \nabla \left( \frac{1}{\varrho(x)} \right)
        \cdot \nabla p_u(x) + \Delta p_u(x) \, .
\end{eqnarray*}
For the first part, we further have
\[
  \varrho(x) \frac{\partial}{\partial x_k} \frac{1}{\varrho(x)} =
  \varrho(x) \, \left[ - \frac{1}{[\varrho(x)]^2}
    \frac{\partial}{\partial x_k} \varrho(x)\right] =
  - \frac{1}{\varrho(x)} \frac{\partial}{\partial x_k} \varrho(x),
\]
and hence,
\[
\varrho(x) \, \nabla \left( \frac{1}{\varrho(x)} \right)
\cdot \nabla p_u(x) = - \frac{1}{\varrho(x)} \, \nabla \varrho(x)
\cdot \nabla p_u(x) ,
\]
i.e.,
\[
  \Delta p_u(x) = \varrho(x) \Delta u(x) +
  \frac{1}{\varrho(x)} \, \nabla \varrho(x)
\cdot \nabla p_u(x) .
\]
When taking the square and applying H\"older's inequality we obtain
\begin{eqnarray*}
  [\Delta p_u(x)]^2
  & \leq & 2 \, [ \varrho(x) \Delta \overline{u}(x)]^2 +
           2 \, \frac{1}{[\varrho(x)]^2} \, \Big[ \nabla \varrho(x)
           \cdot \nabla p_u(x) \Big]^2 \\
  & \leq & 2 \, [\varrho(x) \Delta \overline{u}(x)]^2 +
           2 \, \frac{1}{[\varrho(x)]^2} \, |\nabla \varrho(x)|^2 \,
           |\nabla p_u(x)|^2 ,
\end{eqnarray*}
and therefore
\begin{eqnarray*}
  \| \Delta p_u \|^2_{L^2(\Omega)}
  & \leq & 2 \, \| \varrho \Delta \overline{u} \|^2_{L^2(\Omega)}
           + 2 \, \int_\Omega \frac{|\nabla \varrho(x)|^2}{[\varrho(x)]^2}
           \, |\nabla p_u(x)|^2 \, dx \\
    & \leq & 2 \, \| \varrho \Delta \overline{u} \|^2_{L^2(\Omega)}
             + 2 \, \sup\limits_{x \in \Omega}
             \frac{|\nabla \varrho(x)|^2}{[\varrho(x)]}
             \int_\Omega \frac{1}{[\varrho(x)]}
             \, |\nabla p_u(x)|^2 \, dx
\end{eqnarray*}
follows. Using \eqref{Norm H1 diffusion} we finally obtain
\[
  \| S_\varrho u \|^2_{L^2(\Omega)} \leq 2 \, \| \varrho \Delta u \|^2_{L^2(\Omega)}
  + 2 \, \sup\limits_{x \in \Omega}
             \frac{|\nabla \varrho(x)|^2}{[\varrho(x)]}
             \int_\Omega \varrho(x)] \, |\nabla u(x)|^2 \, dx \, .
\]
When assuming
\begin{equation}\label{Assumption varrho}
  \sup\limits_{x \in \Omega}
  \frac{|\nabla \varrho(x)|^2}{[\varrho(x)]} \leq c_\varrho,
\end{equation}
and combining this with \eqref{abstract regularization L2 H2} we obtain
\begin{equation}\label{regularization diffusion L2 H2}
  \| u_\varrho - \overline{u} \|_{L^2(\Omega)}^2 \leq
  2 \, \| \varrho \Delta \overline{u} \|^2_{L^2(\Omega)}
  + 2 \, c_\varrho \,
  \int_\Omega \varrho(x) \, |\nabla \overline{u}(x)|^2 \, dx \, .
\end{equation}
While for a constant regularization parameter $\varrho(x)=\varrho$
we obviously have $c_\varrho=0$, in the more general situation of a, e.g.,
piecewise linear parameter function $\varrho(x)$ we finally obtain
a similar error estimate as in \eqref{regularization diffusion L2 H1}
when assuming $\overline{u} \in H^1_0(\Omega)$ only. 
{Hence we
restrict our considerations to $\overline{u} \in H^1_0(\Omega)$,
and to less regular target functions
$\overline{u} \in [L^2(\Omega),H^1_0(\Omega)]_s$ for $s\in [0,1)$,
where we can formulate the following results.

\begin{theorem}
  Let $u_\varrho \in H^1_0(\Omega)$ be the unique solution of the
  Schur complement system \eqref{Operator diffusion}, where the regularization
  function $\varrho \in L^\infty(\Omega)$ is assumed to be bounded and
  uniform positive. 
  For $\overline{u} \in H^1_0(\Omega)$, there holds the regularization error estimate
  \begin{equation}\label{final regularization error estimate}
    \| u_\varrho - \overline{u} \|_{L^2(\Omega)}^2 \leq \int_\Omega
    \varrho(x) \, |\nabla \overline{u}(x)|^2 \, dx .
  \end{equation}
\end{theorem}

\noindent
Using space interpolation arguments we can combine the error estimates
\eqref{abstract regularization L2 L2} and
\eqref{final regularization error estimate} to derive related estimates
also for $\overline{u} \in [L^2(\Omega),H^1_0(\Omega)]_s$ for some
$s \in (0,1)$. Since the right hand side in
\eqref{final regularization error estimate} defines a weighted norm
in $H^1_0(\Omega)$, we can consider the eigenvalue problem
  \[
    - \mbox{div} \left[ \frac{1}{\varrho(x)} \, \nabla v(x) \right] =
      \lambda \, v(x) \quad \mbox{for} \; x \in \Omega, \quad
      v(x) = 0 \quad \mbox{for} \; x \in \partial \Omega , \quad
      \| v \|_{L^2(\Omega)}= 1 ,
  \]
  where the eigenfunctions $\{ v_i \}_{i \in {\mathbb{N}}}$ form an
  orthonormal basis in $L^2(\Omega)$, and the
  eigenvalues $\lambda_i=\lambda_i(\varrho) \in {\mathbb{R}}_+$
  are positive and tend to infinity as $i \to \infty$. Hence we can
  define, for $s \in [0,1]$, the weighted Sobolev norms
  \[
    \| u \|^2_{H^s_\varrho(\Omega)} :=
    \sum\limits_{i=1}^\infty \lambda_i^s \, u_i^2, \quad
    u_i = \langle u , v_i \rangle_{L^2(\Omega)} .
  \]
  Now we can formulate the final result of this section.

  \begin{theorem}
    For $\overline{u} \in [L^2(\Omega),H^1_0(\Omega)]_s$ 
    with
    some
    $s \in [0,1]$, 
    there holds the error estimate
    \begin{equation}\label{final regularization estimate s}
      \| u_\varrho - \overline{u} \|_{L^2(\Omega)} \leq
      \| \overline{u} \|_{H^s_\varrho(\Omega)}  .
    \end{equation}
  \end{theorem}
  
  \begin{proof}
    Recall that \eqref{final regularization estimate s} for $s=0$
    is \eqref{abstract regularization L2 L2}, while we have
    \eqref{final regularization error estimate} for $s=1$. Hence
    the assertion for $s \in (0,1)$ follows from interpolation.
    Under additional assumptions on $\varrho$, one can give a more
    explicit representation of the norm for $H_\varrho^s(\Omega)$ showing the 
    explicit dependency on the powers of $\varrho$, see
  {\rm \cite[Theorem 3.4.3]{Triebel:1978}}. 
  \end{proof}

%
%
\section{Finite element error estimates}
\label{sec:FiniteElementErrorEstimates}
Let $V_h = S_h^1(\Omega) \cap H^1_0(\Omega) =
\mbox{span} \{ \varphi_k \}_{k=1}^M$ be the finite element space of
piecewise linear and continuous basis functions $\varphi_k$, which are
defined with respect to some admissible locally quasi-uniform
decomposition of the computational domain $\Omega$ into $N$
simplicial shape-regular finite elements $\tau_\ell$ of local mesh size
$h_\ell = \Delta_\ell^{1/n}$, where $\Delta_\ell$ is the volume of the
finite element $\tau_\ell$, $\ell = 1,\ldots,N$.
As regularization we consider the piecewise constant function
\begin{equation}\label{pw constant regularization}
\varrho(x) = h_\ell^2 \quad \mbox{for} \; x \in \tau_\ell .
\end{equation}
The Galerkin variational formulation of the abstract operator equation
\eqref{Abstract operator equation} is to find $u_{\varrho h} \in V_h$ such that
\begin{equation}\label{FEM abstract}
\langle S_\varrho u_{\varrho h} , v_h \rangle_\Omega +
\langle u_{\varrho h} , v_h \rangle_{L^2(\Omega)} =
\langle \overline{u} , v_h \rangle_{L^2(\Omega)} \quad \mbox{for all} \;
v_h \in V_h .
\end{equation}
Using standard arguments we conclude Cea's lemma,
\begin{equation}\label{Cea Lemma}
\| u_\varrho - u_{\varrho h} \|_{S_\varrho}^2 +
\|u_\varrho - u_{\varrho h} \|^2_{L^2(\Omega)} \leq
\inf\limits_{v_h \in V_h } \Big[
\| u_\varrho - v_h \|_{S_\varrho}^2 +
\|u_\varrho - v_h \|^2_{L^2(\Omega)} \Big],
\end{equation}
from which we further obtain
\begin{eqnarray}\label{Error general}
  \| u_\varrho - u_{\varrho h} \|^2_{L^2(\Omega)}
  & \leq & 2 \, \Big[
           \| u_\varrho - \overline{u} \|^2_{S_\varrho} +
           \| u_\varrho - \overline{u} \|^2_{L^2(\Omega)} \\
  & & \hspace*{1.5cm} \nonumber
  + \inf\limits_{v_h \in V_h} \Big(
      \| \overline{u} - v_h \|_{S_\varrho}^2 +
      \| \overline{u} - v_h \|^2_{L^2(\Omega)} \Big) \Big] . 
\end{eqnarray}
When assuming $\overline{u} \in H^1_0(\Omega)$, and using the
regularization error estimates \eqref{abstract regularization H1 H1} and
\eqref{abstract regularization L2 H1}, this gives
\[
  \| u_\varrho - u_{\varrho h} \|^2_{L^2(\Omega)} \leq 2 \, \Big[
           2 \, \| \overline{u} \|^2_{S_\varrho} 
  + \inf\limits_{v_h \in V_h} \Big(
      \| \overline{u} - v_h \|_{S_\varrho}^2 +
      \| \overline{u} - v_h \|^2_{L^2(\Omega)} \Big) \Big] . 
\]
Let $\Pi_h \overline{u} \in V_h$ be a quasi-interpolation of
$\overline{u} \in H^1_0(\Omega)$, e.g., using the Scott--Zhang operator
$\Pi_h$, see, e.g., \cite{LLSY2:BrennerScott:2008}, and satisfying the error
estimates
\begin{equation}\label{error Scott Zhang L2}
  \| \overline{u} - \Pi_h \overline{u} \|_{L^2(\tau_\ell)}
  \leq c \, h_\ell \, \| \nabla \overline{u} \|_{L^2(\omega_\ell)},
\end{equation}
and
\begin{equation}\label{error Scott Zhang H1}
  \| \nabla( \overline{u} - \Pi_h \overline{u}) \|_{L^2(\tau_\ell)}
  \leq c \, \| \nabla \overline{u} \|_{L^2(\omega_\ell)}.
\end{equation}
Here, $\omega_\ell := \{ \tau_j: \overline{\tau}_\ell\cap \overline{\tau}_j
\neq \emptyset\}$ is the simplex patch of $\tau_\ell$.
Then we can estimate, using \eqref{Norm H1 diffusion} and
\eqref{error Scott Zhang H1},
\begin{eqnarray*}
  \| \overline{u} - \Pi_h \overline{u} \|_{S_\varrho}^2
  & \leq & \int_\Omega \varrho(x) \,
           |\nabla (\overline{u}-\Pi_h\overline{u}(x))|^2 \, dx \\
  & = & \sum_{\ell=1}^N h_\ell^2 \int_{\tau_\ell}
        |\nabla (\overline{u}(x)-\Pi_h\overline{u}(x))|^2 \, dx \\
  & = & \sum_{\ell=1}^N h_\ell^2 \,
        \| \nabla (\overline{u}-\Pi_h\overline{u}) \|_{L^2(\tau_\ell)}^2
        \, \leq \, c \, \sum_{\ell=1}^N h_\ell^2 \,
        \|\nabla \overline{u} \|_{L^2(\omega_\ell)}^2.
\end{eqnarray*}
Moreover, using \eqref{error Scott Zhang L2}, we also have
\[
  \| \overline{u} - \Pi_h \overline{u} \|_{L^2(\Omega)}^2 =
  \sum_{\ell=1}^N \| \overline{u} - \Pi_h \overline{u} \|_{L^2(\tau_\ell)}^2
  \leq c \, \sum_{\ell=1}^N h_\ell^2 \,
  \| \nabla \overline{u} \|_{L^2(\omega_\ell)}^2. 
\]
Together with \eqref{Norm H1 diffusion} we then obtain
\begin{equation}\label{error diffusion urho}
  \| u_\varrho - u_{\varrho h} \|^2_{L^2(\Omega)} \leq
  c \, \sum_{\ell=1}^N h_\ell^2 \, \| \nabla \overline{u} \|_{L^2(\omega_\ell)}^2,
\end{equation}
and with \eqref{regularization diffusion L2 H1} this finally gives
\begin{equation}\label{error diffusion target}
  \| u_{\varrho h} - \overline{u} \|_{L^2(\Omega)}^2 \leq
  c \, \sum_{\ell=1}^N h_\ell^2 \,
  \|\nabla \overline{u} \|_{L^2(\omega_\ell)}^2.
\end{equation}
The variational formulation \eqref{FEM abstract} requires, for any
given $u \in H^1_0(\Omega)$, the evaluation of
$S_\varrho u = B^* A_{1/\varrho}^{-1} B u = B^* p_u$, where
$ p_u = A_{1/\varrho}^{-1} Bu \in H^1_0(\Omega)$ is the unique solution
of the variational formulation
\begin{equation}\label{Def pu}
  \int_\Omega \frac{1}{\varrho(x)} \, \nabla p_u(x) \cdot
  \nabla v(x) \, dx = \int_\Omega \nabla u(x) \cdot \nabla v(x) \, dx
  \quad \mbox{for all} \; v \in H^1_0(\Omega).
\end{equation}
Hence we can define the approximate solution $p_{uh} \in V_h$
satisfying
\begin{equation}\label{Def puh}
  \int_\Omega \frac{1}{\varrho(x)} \, \nabla p_{uh}(x) \cdot
  \nabla v_h(x) \, dx = \int_\Omega \nabla u(x) \cdot \nabla v_h(x) \, dx
  \quad \mbox{for all} \; v_h \in V_h,
\end{equation}
and therefore we can introduce an approximation
$\widetilde{S}_\varrho u = B^* p_{uh}$ of $S_\varrho u = B^* p_u$.
Instead of \eqref{FEM abstract} we now consider the perturbed variational
formulation to find $\widetilde{u}_{\varrho h} \in V_h$ such that
\begin{equation}\label{FEM diffusion}
  \langle \widetilde{S}_\varrho \widetilde{u}_{\varrho h} , v_h \rangle_\Omega
  +
  \langle \widetilde{u}_{\varrho h} , v_h \rangle_{L^2(\Omega)}
  =
  \langle \overline{u} , v_h \rangle_{L^2(\Omega)} \quad
  \mbox{for all} \; v_h \in V_h .
\end{equation}
Unique solvability of \eqref{FEM diffusion} follows since the
stiffness matrix of $\widetilde{S}_\varrho$ is positive semi-definite,
while the mass matrix related to the inner product in $L^2(\Omega)$
is positive definite.

\begin{lemma}
  Let $\widetilde{u}_{\varrho h} \in V_h$ be the unique solution of the
  perturbed variational formulation \eqref{FEM diffusion}. Then there
  holds the error estimate
  \begin{equation}\label{error utilde}
    \| \widetilde{u}_{\varrho h} - u_{\varrho h} \|_{L^2(\Omega)} \leq
    c \, \sum\limits_{\ell=1}^N h_\ell^2 \,
    \| \nabla \overline{u} \|^2_{L^2(\omega_\ell)} .
  \end{equation}
\end{lemma}

\begin{proof}
  The difference of the variational formulations
  \eqref{FEM abstract} and \eqref{FEM diffusion} first gives the
  Galerkin orthogonality
  \[
    \langle S_\varrho u_{\varrho h} - \widetilde{S}_\varrho
    \widetilde{u}_{\varrho h} , v_h \rangle_\Omega +
    \langle u_{\varrho h} - \widetilde{u}_{\varrho h},v_h \rangle_{L^2(\Omega)} = 0
    \quad \mbox{for all} \; v_h \in V_h ,
  \]
  which can be written as
  \[
    \langle \widetilde{S}_\varrho (\widetilde{u}_{\varrho h} - u_{\varrho h}),
    v_h \rangle_\Omega + \langle \widetilde{u}_{\varrho h} - u_{\varrho h},
    v_h \rangle_{L^2(\Omega)} = \langle
    (S_\varrho - \widetilde{S}_\varrho) u_{\varrho h} , v_h \rangle_\Omega
    \quad \mbox{for all} \; v_h \in V_h .
  \]
  When chosing
  $v_h = \widetilde{u}_{\varrho h} - u_{\varrho h} \in V_h$, and using
  $\langle \widetilde{S}_\varrho u , u \rangle_\Omega \geq 0$ for all
  $u \in H^1_0(\Omega)$, this gives
  \begin{eqnarray*}
    \| \widetilde{u}_{\varrho h} - u_{\varrho h} \|^2_{L^2(\Omega)}
    & \leq & \langle (S_\varrho - \widetilde{S}_\varrho) u_{\varrho h} ,
             \widetilde{u}_{\varrho h} - u_{\varrho h} \rangle_\Omega \\[2mm]
    & & \hspace*{-2cm} 
        = \, \langle B^* (p_{u_{\varrho h}} - p_{u_{\varrho h}h}) ,
          \widetilde{u}_{\varrho h} - u_{\varrho h} \rangle_\Omega \\[2mm]
    & & \hspace*{-2cm} = \,
        \int_\Omega \nabla (p_{u_{\varrho h}} - p_{u_{\varrho h}h}) \cdot
          \nabla  (\widetilde{u}_{\varrho h} - u_{\varrho h} ) \, dx \\
    & & \hspace*{-2cm} \leq \, \left( \int_\Omega \frac{1}{\varrho} \,
        |\nabla (p_{u_{\varrho h}} - p_{u_{\varrho h}h})|^2
        \, dx \right)^{1/2} \left( \int_\Omega \varrho \,
             |\nabla (\widetilde{u}_{\varrho h} - u_{\varrho h} )|^2
             \, dx \right)^{1/2}.
  \end{eqnarray*}
  Using \eqref{pw constant regularization} and an inverse inequality
  locally, we further have
  \begin{eqnarray*}
    \int_\Omega \varrho \,
    |\nabla (\widetilde{u}_{\varrho h} - u_{\varrho h} )|^2 \, dx
    & = &
    \sum\limits_{\ell=1}^N h_\ell^2 \,
    \| \nabla (\widetilde{u}_{\varrho h} - u_{\varrho h} )
          \|_{L^2(\tau_\ell)}^2 \\
    & \leq & c \, \sum\limits_{\ell=1}^N
             \| \widetilde{u}_{\varrho h} - u_{\varrho h} \|_{L^2(\tau_\ell)}^2
             \, = \, c \,
             \| \widetilde{u}_{\varrho h} - u_{\varrho h} \|_{L^2(\Omega)}^2 ,
  \end{eqnarray*}
  and hence,
  \begin{eqnarray*}
    \| \widetilde{u}_{\varrho h} - u_{\varrho h} \|_{L^2(\Omega)}^2
    & \leq & c \, \int_\Omega \frac{1}{\varrho} \,
             |\nabla (p_{u_{\varrho h}} - p_{u_{\varrho h}h})|^2 \, dx \\
    & & \hspace*{-3cm} = \,
        c \, \langle A_{1/\varrho} (p_{u_{\varrho h}} - p_{u_{\varrho h}h}),
          p_{u_{\varrho h}} - p_{u_{\varrho h}h} \rangle_\Omega \, = \,
          c \, \| p_{u_{\varrho h}} - p_{u_{\varrho h}h} \|^2_{A_{1/\varrho}},
  \end{eqnarray*}
  i.e., 
  \begin{eqnarray*}
    \| \widetilde{u}_{\varrho h} - u_{\varrho h} \|_{L^2(\Omega)}
    & \leq & c \, \| p_{u_{\varrho h}} - p_{u_{\varrho h}h} \|_{A_{1/\varrho}} \\
    & & \hspace*{-3cm} \leq \, c \, \Big[
             \| p_{u_{\varrho h}} - p_{\overline{u}} \|_{A_{1/\varrho}} +
             \| p_{\overline{u}h} - p_{u_{\varrho h}h} \|_{A_{1/\varrho}} +
             \| p_{\overline{u}} - p_{\overline{u}h} \|_{A_{1/\varrho}} \Big] .
  \end{eqnarray*}
  Note that we have
  \[
    \int_\Omega \frac{1}{\varrho} \, \nabla (p_{u_{\varrho h}} - p_{\overline{u}})
    \cdot \nabla v \, dx =
    \int_\Omega \, \nabla (u_{\varrho h} - \overline{u})
    \cdot \nabla v \, dx \quad \mbox{for all} \; v \in H^1_0(\Omega)
  \]
  and
  \[
    \int_\Omega \frac{1}{\varrho} \, \nabla (p_{u_{\varrho h}h} - p_{\overline{u} h})
    \cdot \nabla v_h \, dx =
    \int_\Omega \, \nabla (u_{\varrho h} - \overline{u})
    \cdot \nabla v_h \, dx \quad \mbox{for all} \; v_h \in V_h .
  \]
  Hence we conclude
  \[
    \| p_{u_{\varrho h}} - p_{\overline{u}} \|_{A_{1/\varrho}}^2 \leq
    \int_\Omega \varrho \, |\nabla (u_{\varrho h} - \overline{u})|^2 \, dx,
  \]
  as well as
  \[
    \| p_{u_{\varrho h}h} - p_{\overline{u} h} \|_{A_{1/\varrho}}^2 \leq
    \int_\Omega \varrho \, |\nabla (u_{\varrho h} - \overline{u})|^2 \, dx .
  \]
  We can further obtain, inserting the Scott-Zhang interpolation
  $\Pi_h \overline{u}$, 
  \begin{eqnarray*}
    \int_\Omega \varrho \, |\nabla (u_{\varrho h} - \overline{u})|^2 \, dx
    & \leq & 2 \, \left[ \int_\Omega \varrho \,
             |\nabla (u_{\varrho h} - \Pi_h \overline{u})|^2 \, dx +
             \int_\Omega \varrho \, |\nabla (\overline{u} -
             \Pi_h \overline{u})|^2\, dx \right] \\
    & \leq & 2 \, \left[ \int_\Omega \varrho \,
             |\nabla (u_{\varrho h} - \Pi_h \overline{u})|^2 \, dx +
             c \, \sum_{\ell=1}^N h_\ell^2 \,
             \|\nabla \overline{u} \|_{L^2(\omega_\ell)}^2 \right].
  \end{eqnarray*}
  Using an inverse inequality locally, we further estimate the first term by 
  \begin{eqnarray*}
    \int_\Omega \varrho \, |\nabla (u_{\varrho h} - \Pi_h \overline{u})|^2 \, dx
    & = & \sum_{\ell=1}^N h_\ell^2 \, \|\nabla (u_{\varrho h} -
          \Pi_h \overline{u}) \|_{L^2(\tau_\ell)}^2 \\
    & & \hspace*{-3cm} \leq \, c \, \sum_{\ell=1}^N
             \|u_{\varrho h} - \Pi_h\overline u \|_{L^2(\tau_\ell)}^2 
        \, = \, \| u_{\varrho h} - \Pi_h \overline{u} \|_{L^2(\Omega)}^2 \\
    & & \hspace*{-3cm} \leq \, 2 \, \left[ \|u_{\varrho h}-\overline{u}
        \|_{L^2(\Omega)}^2 +
        \| \overline u-\Pi_h\overline{u}\|_{L^2(\Omega)}^2 \right] \, \leq \,
	c \, \sum_{\ell=1}^N h_\ell^2 \,
        \|\nabla \overline u\|_{L^2(\omega_\ell)}^2 .
  \end{eqnarray*}
  Recall that $p_{\overline{u}} \in H^1_0(\Omega)$ solves
  \[
    \int_\Omega \frac{1}{\varrho(x)} \, \nabla p_{\overline{u}} \cdot
    \nabla v(x) \, dx = \int_\Omega \nabla \overline{u}(x)
    \cdot \nabla v(x) \, dx \quad \mbox{for all} \; v \in H^1_0(\Omega),
  \]
  while $p_{\overline{u} h} \in V_h \subset H^1_0(\Omega)$ solves
    \[
    \int_\Omega \frac{1}{\varrho(x)} \, \nabla p_{\overline{u} h} \cdot
    \nabla v_h(x) \, dx = \int_\Omega \nabla \overline{u}(x)
    \cdot \nabla v_h(x) \, dx \quad \mbox{for all} \; v_h \in V_h.
  \]
  Hence we conclude the Galerkin orthogonality
  \[
    \int_\Omega \frac{1}{\varrho(x)} \, \nabla
    (p_{\overline{u}}(x) - p_{\overline{u} h}(x)) \cdot \nabla v_h(x) \, dx
    = 0 \quad \mbox{for all} \; v_h \in V_h,  
  \]
  and
  \[
    \| p_{\overline{u}} - p_{\overline{u} h} \|_{A_{1/\varrho}} \leq
    \| p_{\overline{u}} \|_{A_{1/\varrho}} .
  \]
  Now the assertion follows from
  \[
    \| p_{\overline{u}} \|^2_{A_{1/\varrho}} =
    \int_\Omega \frac{1}{\varrho(x)} \, |\nabla p_{\overline{u}}(x)|^2 \, dx
    \leq \int_\Omega \varrho(x) \, |\nabla \overline{u}(x)|^2 \, dx
    \, \leq \, \sum_{\ell=1}^N h_\ell^2 \,
    \|\nabla \overline{u}\|_{L^2(\omega_\ell)}^2.
  \]  
\end{proof}

\noindent
Now we are in the position to state the main results of this paper.

\begin{theorem}
  Let $\widetilde{u}_{\varrho h} \in V_h \subset H^1_0(\Omega)$ be the
  unique solution of the perturbed variational formulation
  \eqref{FEM diffusion} where the regularization function $\varrho(x)$
  is given as in \eqref{pw constant regularization}, and where the
  underlying finite element mesh is assumed to be locally
  quasi-uniform. Then there holds the error estimate
\begin{equation}\label{final error}
  \| \widetilde{u}_{\varrho h} - \overline{u} \|^2_{L^2(\Omega)} \leq
  c \, \sum\limits_{\ell=1}^N h_\ell^2 \, \| \nabla \overline{u}
  \|^2_{L^2(\tau_\ell)} =
  c \int_\Omega \varrho(x) \, |\nabla \overline{u}(x)|^2 \, dx \, .
\end{equation}
\end{theorem}

\begin{proof}
The estimate \eqref{final error} follows from
the finite element error estimates
\eqref{error diffusion target} and \eqref{error utilde},
since the finite element mesh is assumed to be locally quasi-uniform.
\end{proof}

\begin{theorem}
  Similar as before we  also have the error estimate
  \[
    \| \widetilde{u}_{\varrho h} - \overline{u} \|_{L^2(\Omega)} \leq
    c \, \| \overline{u} \|_{H^s_\varrho(\Omega)},
  \]
  when assuming $\overline{u} \in [L^2(\Omega),H^1_0(\Omega)]_s$
  for some $s \in [0,1]$.
\end{theorem}

\noindent
The perturbed Galerkin finite element formulation
\eqref{FEM diffusion} can be written as coupled system to find
$(\widetilde{u}_{\varrho h},\widetilde{p}_{\varrho h}) \in V_h \times V_h$
such that
\begin{equation}\label{FEM primal}
  \int_\Omega \frac{1}{\varrho(x)} \, \nabla \widetilde{p}_{\varrho h}(x) \cdot
  \nabla v_h(x) \, dx + \int_\Omega \nabla \widetilde{u}_{\varrho h}(x)
  \cdot \nabla v_h(x) \, dx = 0
\end{equation}
for all $v_h \in V_h$, and
\begin{equation}\label{FEM adjoint}
  \int_\Omega \widetilde{u}_{\varrho h}(x) \, q_h(x) \, dx
  - \int_\Omega \nabla \widetilde{p}_{\varrho h}(x) \cdot \nabla q_h(x) \, dx
   =
  \int_\Omega \overline{u}(x) \, q_h(x) \, dx
\end{equation}
for all $q_h \in V_h$.
Note that this system corresponds to the finite element discretization of the
coupled variational formulation \eqref{VF primal} and \eqref{VF adjoint}.

The finite element variational formulation \eqref{FEM primal} and
\eqref{FEM adjoint} is equivalent to a coupled linear system of
algebraic equations,
\begin{equation}\label{coupled LGS}
  K_{\varrho h} \underline{p} + K_h \underline{u} = \underline{0}, \quad
  M_h \underline{u} - K_h \underline{p} = \underline{f},
\end{equation}
where we use the standard finite element stiffness and mass matrices
defined as
\begin{eqnarray*}
  K_h[j,k]
  & = & \int_\Omega \nabla \varphi_k(x) \cdot \nabla \varphi_j(x) \, dx, \\
  K_{\varrho h}[j,k]
  & = & \int_\Omega \frac{1}{\varrho(x)} \, \nabla \varphi_k(x)
  \cdot \nabla \varphi_j(x) \, dx, \\
  M_h[j,k]
  & = & \int_\Omega \varphi_k(x) \, \varphi_j(x) \, dx
\end{eqnarray*}
for $j,k=1,\ldots,M$, and the entries of the load vector
\[
  f_j = \int_\Omega \overline{u}(x) \, \varphi_j(x) \, dx \quad
  \mbox{for} \; j=1,\ldots,M.
\]
Since the finite element stiffness matrix $K_{\varrho h}$ is invertible,
we can eliminate the adjoint $\underline{p}$ to end up with the
Schur complement system
\begin{equation}\label{Schur LGS}
  \Big[ M_h + K_h K_{\varrho h}^{-1} K_h \Big] \underline{u} =
  \underline{f} .
\end{equation}
Since all involved stiffness and mass matrices are symmetric and positive
definite, unique solvability of the Schur complement system and therefore
of the Galerkin variational formulation \eqref{FEM primal} and
\eqref{FEM adjoint} follows.
%
%
%

%
%
%
%
%
\section{Numerical results}
\label{sec:Numerical results}
%
%
\subsection{Convergence studies}
\label{subsec:Convergence studies}
As a first numerical example we consider the two-dimensional
domain $\Omega = (0,1)^2$, and the discontinuous target function
\[
\overline{u}_{2D}(x) =
\left \{ \begin{array}{ccl}
	1 & & \mbox{for} \; x \in (0.25,0.75)^2, \\[1mm]
	0 & & \mbox{else.}
\end{array} \right.                      
\]
The initial mesh consists of $32$ triangular finite elements       
and $9$ degrees of freedom, see Figure~\ref{fig:meshes}.

\begin{figure}[ht]
	\centering
	\includegraphics[width=5cm]{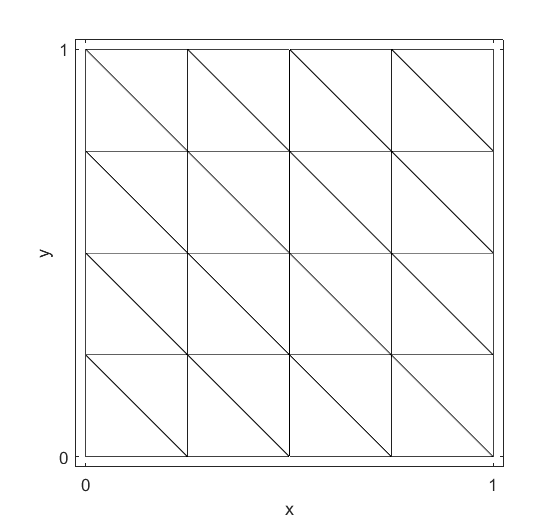}
	\includegraphics[width=5cm]{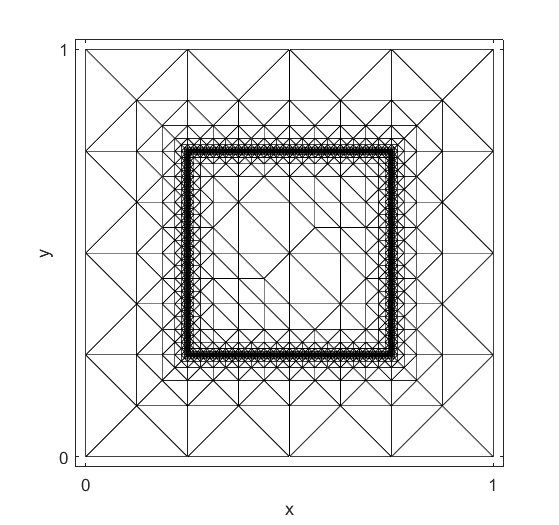}
	\caption{Initial mesh and adaptively refined mesh at level 14
          for $\Omega = (0,1)^2$ and $\overline{u}_{2D}$.}
	\label{fig:meshes}
\end{figure}

%
%

For a given mesh we compute the approximate solution
$\widetilde{u}_{\varrho h}$, the global error
\[
\eta := \| \widetilde{u}_{\varrho h} - \overline{u} \|_{L^2(\Omega)},
\]
and the local error indicators
\[
\eta_\ell := \| \widetilde{u}_{\varrho h} - \overline{u} \|_{L^2(\tau_\ell)},
\quad
\eta^2 = \sum\limits_{\ell=1}^N \eta_\ell^2 .
\]
We mark all elements $\tau_\ell$ when
\[
\eta_\ell > \theta \max\limits_{\ell=1,\ldots,N} \eta_\ell 
\]
is satisfied, with $\theta = 0.5$. After $14$ refinement steps we obtain
the mesh as shown in Fig.~\ref{fig:meshes} with $1.310.444$ finite
elements
and $655.215$ degrees of freedom. According to the final error estimate
\eqref{final error} we expect a linear order of convergence.
The numerical results are shown in Fig.~\ref{fig:discont-convergence_2D}
where in addition to the present approach we also present the
convergence results when considering energy regularization
\cite{UL:LangerSteinbachYang:2022a} with the optimal choice $\varrho=h^2$,
and the regularization in $L^2(\Omega)$ with $\varrho=h^4$, see
\cite{LLSY2:LangerLoescherSteinbachYang:2022a}.
As expected, we observe a linear order of convergence 
when using the variable energy regularization in the adaptive version
described above, while both the energy regularization
and the regularization in $L^2(\Omega)$ almost coincide with half the
order of convergence. For a comparison of the different approaches,
see also the computed
states as shown in Fig. \ref{fig:solutions}.

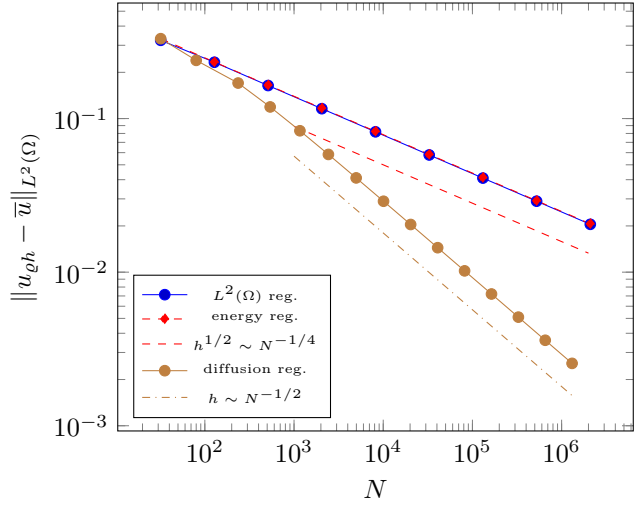
\begin{figure}
  \centering
  \begin{tikzpicture}
    \begin{axis}[
      xmode = log,
      ymode = log,
      xlabel=$N$,
      ylabel=$\| u_{\varrho h}- \overline{u} \|_{L^2(\Omega)}$,
      legend pos=south west, legend style={font=\tiny}]
      \addplot table [col sep=
      &, y=err, x=N]{tab_discontinuous_l2reg_h4.dat};
      \addlegendentry{$L^2(\Omega)$ reg.}
      \addplot[dashed, mark=diamond*, color=red]
      table [col sep=
      &, y=err, x=N]{tab_discontinuous_energ-reg_h2.dat};
      \addlegendentry{energy reg.}
      \addplot[
      domain = 1000:2000000,
      samples = 10,
      dashed, thin,red,] {0.5*x^(-0.25)};
      \addlegendentry{$h^{1/2}\sim N^{-1/4}$}
      \addplot[mark=*, color=brown] table [col sep=
      &, y=err, x=N]{tab_discontinuous_diffusion-reg_h2elem.dat};
      \addlegendentry{diffusion reg.}
      \addplot[
      domain = 1000:1300000,samples = 10,
      dash dot,thin,brown,] {1.8*x^(-0.5)};
      \addlegendentry{$h\sim N^{-1/2}$}
    \end{axis}
  \end{tikzpicture}

  \caption{Convergence plots for $\overline u_{2D}$ 
      choosing $\varrho =h^4$ for the $L^2(\Omega)$ regularization,
      $\varrho=h^2$ for the energy regularization, and
      $\varrho(x)=h_\ell^2$ for $x\in \tau_\ell$ for the
      diffusion regularization.}
  \label{fig:discont-convergence_2D}
\end{figure}

\begin{figure}[ht]
	\centering
	\begin{subfigure}[b]{0.3\textwidth}
		\includegraphics[width=\textwidth]{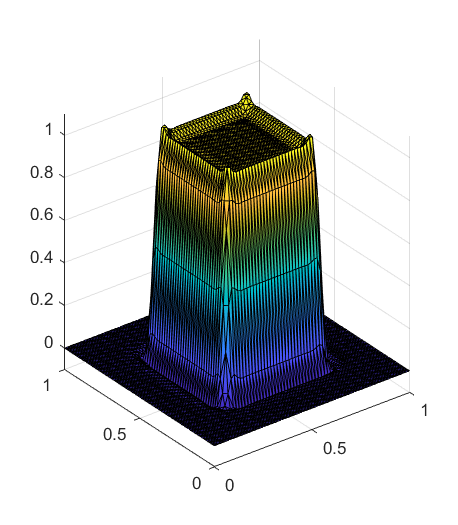}
		\caption{$L^2(\Omega)$ regularization}
	\end{subfigure}
	\hfill
	\begin{subfigure}[b]{0.3\textwidth}
		\includegraphics[width=\textwidth]{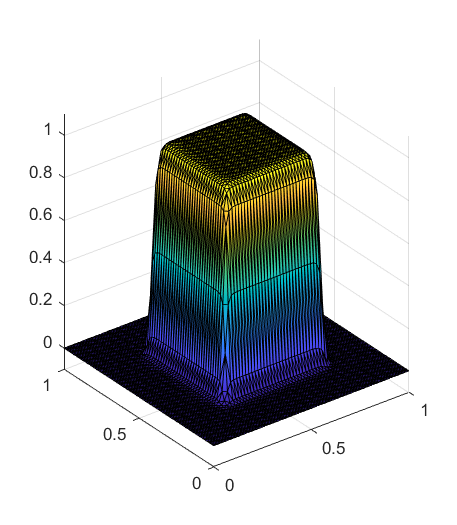}
		\caption{Energy regularization}
	\end{subfigure}
	\hfill
	\begin{subfigure}[b]{0.3\textwidth}
		\includegraphics[width=\textwidth]{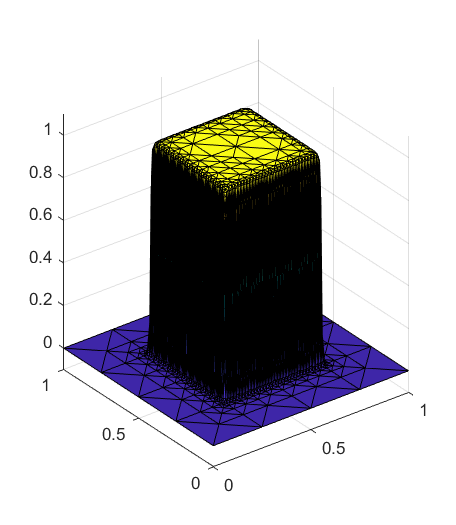}
		\caption{Diffusion regularization}
	\end{subfigure}
	\caption{Solution $\widetilde{u}_{\varrho h}$ for the three different
          regularization approaches, with 8192 finite elements (level 4)
          using a uniform refinement strategy for $L^2(\Omega)$ and energy
          regularization, and with 4972 finite elements (level 6) for an
          adaptive refinement for the diffusion regularization.}
	\label{fig:solutions}
\end{figure}

Next we consider the three-dimensional domain $\Omega=(0,1)^3$, and
the target function
\[
  \overline{u}_{3D}(x) =
  \begin{cases}
    1 & \mbox{for} \; x\in (0.25,0.75)^3, \\
    0 & \mbox{else}. 
  \end{cases}
\]
As shown in Fig.~\ref{fig:discont-convergence_3D} we still observe
a $h^{1/2}$ convergence for a uniform refinement 
in the case of
both the $L^2(\Omega)$
and energy regularizations, but a $h^{3/4}$ convergence for the adaptive
diffusion approach, where $h = N^{-1/3}$.
       
\begin{figure}
  \centering
  \begin{tikzpicture}
    \begin{axis}[
      xmode = log,
      ymode = log,
      xlabel=$N$,
      ylabel=$\| u_{\varrho h}- \overline{u} \|_{L^2(\Omega)}$,
      legend pos=south west,
      legend style={font=\tiny}]

      \addplot table [col sep=
      &, y=err, x=N]{tab_discontinuous_l2-reg3D_h4.dat};
      \addlegendentry{$L^2(\Omega)$ reg.}

      \addplot[mark=diamond*, color=red] table [col sep=
      &, y=err, x=N]{tab_discontinuous_energ-reg3D_h2.dat};
      \addlegendentry{energy reg.}


      \addplot[
      domain = 50000:1000000,
      samples = 10,
      dashed,thin,red,] {0.61*x^(-1/6)};
      \addlegendentry{$h^{1/2}\sim N^{-1/6}$}

      \addplot[mark=*, color=brown] table [col sep=
      &, y=err, x=N]{tab_discontinuous_diffusion-reg3D_h2elem.dat};
      \addlegendentry{diffusion reg.}


      \addplot[
      domain = 50000:1000000,
      samples = 10,
      dash dot,thin,brown,] {0.95*x^(-0.75/3)};
      \addlegendentry{$h^{0.75}\sim N^{-0.75/3}$}
    \end{axis}
  \end{tikzpicture}
  \caption{Convergence plots for $\overline u_{3D}$
      choosing $\varrho =h^4$ for the $L^2(\Omega)$ regularization,
      $\varrho=h^2$ for the energy regularization, and
      $\varrho(x)=h_\ell^2$ for $x\in \tau_\ell$ for the diffusion
      regularization.}  \label{fig:discont-convergence_3D}
\end{figure}
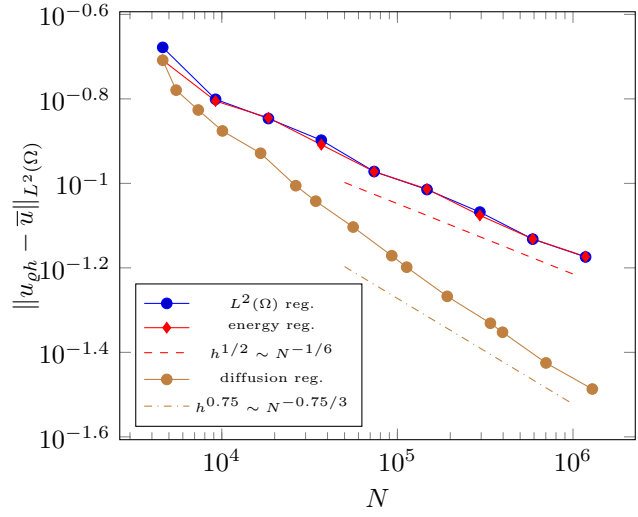

To explain the different convergence behaviour for the adaptive refinement
in two and three space dimensions, we
first consider the 2D case for
the example $\overline{u}_{2D}\in H^{1/2-\varepsilon}(\Omega)$, $\varepsilon > 0$
for a uniform refinement of the triangulation with $N$ triangles
(see Fig.~\ref{fig:meshes} (left)). Let further $m$ denote approximately
the number of elements in each row/column of the mesh grid, i.e.,
$N \sim m \cdot m = m^2$. Due to the regularity of $\overline{u}_{2D}$,
the optimal order of convergence is de facto $1/2$.
Thus, refining all of the $N$
elements uniformly, will lead to a an error reduction of order $h^{1/2}$,
as observed. Now we aim for an adaptive refinement. Since for the
particular test example the singularity of $\overline{u}_{2D}$ is only
along the boundary of the square $(0.25,0.75)^2$, it is sufficient
(after an initializing phase), to refine only
${\mathcal{O}}(m)={\mathcal{O}}(\sqrt{N})$ elements in the
neighborhood of this boundary in order to have the optimal rate of $1/2$.
Note, that for a uniform refinement, the number of elements would grow by a
factor of 4 in each step, for the adaptive scheme, we only refine
${\mathcal{O}}(m)= {\mathcal{O}}(\sqrt{N})$ elements in the neighborhood
of the discontinuity. Therefore, the number of elements only grows by a
factor of 2. Hence, if we adaptively refine ${\mathcal{O}}(N)$ elements,
we can expect an error reduction of order $h$, which is exactly what we
see in the numerical example as given in
Fig.~\ref{fig:discont-convergence_2D}. 

Now let us look at the 3D case.
Here again counting the elements along each edge,
denoted by $m$, we get the relation $N \sim m^3$. For a uniform refinement
we get a convergence rate $h^{1/2}$. In order to get the same rate with an
adaptive scheme, we need to refine at least the elements in the neighborhood
of the boundary of the cube $[0.25,0.75]^3$, where $\overline u_{3D}$ jumps.
Each side of the cube consists of approximately ${\mathcal{O}}(m^2)$
elements. So the whole boundary of the interior cube has approximately
${\mathcal{O}}(m^2)={\mathcal{O}}(N^{2/3})$ neighboring elements. So,
refining $O(N^{2/3})$ elements adaptively gives a rate of $1/2$. Hence,
if we refine ${\mathcal{O}}(N)$ elements adaptively, we might expect
an error reduction of order $h^{3/2 \cdot 1/2}=h^{3/4}$,
which is exactly what we observe in Fig.~\ref{fig:discont-convergence_3D}. 

Finally we consider the one-dimensional domain $\Omega = (0,1)$ and the target
\[
  \overline{u}_{1D}(x) =
  \begin{cases}
    1 & \mbox{for} \; x\in (0.25,0.75) , \\
    0 & \mbox{else}. 
  \end{cases}
\]
For $\overline{u}_{1D}\in H^{1/2-\varepsilon}(\Omega)$, $\varepsilon >0$,
and for uniformly refining $N$ elements, we will get an error reduction of
order $h^{1/2}$. Using an adaptive refinement though, it is enough to refine
exactly $4\sim {\mathcal{O}}(\log(N))$ elements in each step to get the
optimal order of $1/2$. Thus we can expect exponential convergence, which
is also what we observe in the numerical example in
Fig.~\ref{fig:discont-convergence_1D}. 

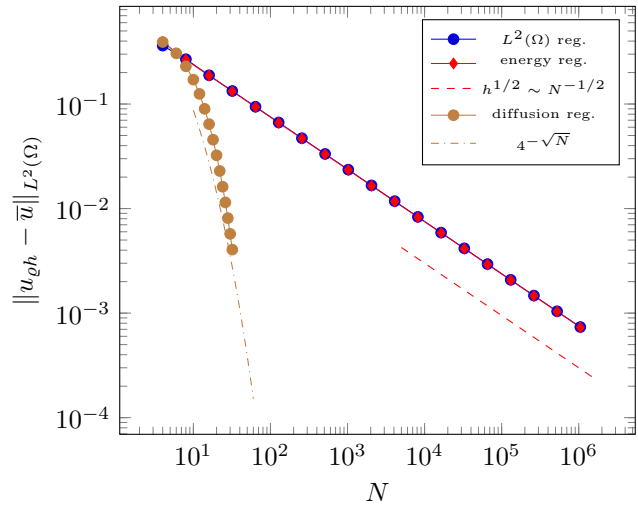
\begin{figure}
  \centering
  \begin{tikzpicture}
    \begin{axis}[
      xmode = log,
      ymode = log,
      xlabel=$N$,
      ylabel=$\| u_{\varrho h}- \overline{u} \|_{L^2(\Omega)}$,
      legend pos=north east,
      legend style={font=\tiny}]

      \addplot table [col sep=
      &, y=err, x=N]{tab_discontinuous_l2-reg1D_h4.dat};
      \addlegendentry{$L^2(\Omega)$ reg.}

      \addplot[mark=diamond*, color=red] table [col sep=
      &, y=err, x=N]{tab_discontinuous_energ-reg1D_h2.dat};
      \addlegendentry{energy reg.}


      \addplot[
      domain = 5000:1500000,
      samples = 10,
      dashed,thin,red,] {0.3*x^(-1/2)};
      \addlegendentry{$h^{1/2}\sim N^{-1/2}$}

      \addplot[mark=*, color=brown] table [col sep=
      &, y=err, x=N]{tab_discontinuous_diffusion-reg1D_h2elem.dat};
      \addlegendentry{diffusion reg.}


      \addplot[
      domain = 10:60,
      samples = 100,
      dash dot,thin,brown,] {7*4^(-sqrt(x))};
      \addlegendentry{$4^{-\sqrt{N}}$}
    \end{axis}
  \end{tikzpicture}
  \caption{Convergence plots for $\overline{u}_{1D}$ 
        choosing $\varrho =h^4$ for the $L^2(\Omega)$ regularization,
        $\varrho=h^2$ for the energy regularization, and
        $\varrho(x)=h_\ell^2$ for $x\in \tau_\ell$ for the
        diffusion regularization.} 
    \label{fig:discont-convergence_1D}
\end{figure}
%
%

\subsection{Control recovering}
\label{subsec:Control recovering} 
Once we have computed an approximation $\widetilde{u}_{\varrho h}$ of the
state $u_\varrho$, we can easily recover the corresponding control
via postprocessing. Using
$A = - \Delta : H^1_0(\Omega) \to H^{-1}(\Omega)$ we can write the state
equation \eqref{primal problem} as $A^{-1} z_\varrho = u_\varrho$, i.e.,
$z_\varrho \in H^{-1}(\Omega)$ solves the variational formulation
\[
  \langle A^{-1} z_\varrho , \psi \rangle_\Omega =
  \langle u_\varrho , \psi \rangle_\Omega \quad
  \mbox{for all} \; \psi \in H^{-1}(\Omega) .
\]
In addition to the finite element space $V_h \subset H^1_0(\Omega)$
of piecewise linear and continuous basis functions, we now define the
ansatz space $Z_H = S_H^0(\Omega) = \text{span}\{\psi_\ell\}_{\ell=1}^{N_H}$ of piecewise constant
basis functions which are defined with respect to some mesh of
mesh size $H \sim h$. Hence we may determine
$\widetilde{z}_{\varrho H} \in Z_H$ as the unique solution of the
Galerkin variational formulation
\[
  \langle A^{-1} z_{\varrho H} , \psi_H \rangle_\Omega =
  \langle \widetilde{u}_{\varrho h} , \psi_H \rangle_\Omega \quad
  \mbox{for all} \; \psi_H \in Z_H.
\]
While we can derive related error estimates using standard arguments,
in general we are not able to evaluate the inverse operator $A^{-1}$.
Hence we need to introduce a suitable approxiation as follows:
For any $z \in H^{-1}(\Omega)$ we define $p_z \in H^1_0(\Omega)$ as the
unique solution of the variational formulation
\[
  \langle A p_z , q \rangle_\Omega =
  \langle \nabla p , \nabla q \rangle_{L^2(\Omega)} =
  \langle z , q \rangle_\Omega \quad \mbox{for all} \; q \in H^1_0(\Omega) .
\]
In addition we determine an approximate solution $p_{zh} \in V_h$ such that
\[
  \langle \nabla p_{zh}, \nabla q_h \rangle_{L^2(\Omega)} =
  \langle z, q_h \rangle_\Omega \quad \mbox{for all} \; q_h \in V_h ,
\]
which defines an approximation
$\widetilde{A}^{-1} z := p_{zh}$ of $p_z = A^{-1}z$.
Hence we finally consider the variational formulation to find
$\widehat{z}_{\varrho H} \in Z_H$ such that
\[
  \langle \widetilde{A}^{-1} \widehat{z}_{\varrho H} , \psi_H \rangle_\Omega =
  \langle \widetilde{u}_{\varrho h} , \psi_H \rangle_\Omega
  \quad \mbox{for all} \; \psi_H \in Z_H .
\]
Unique solvability follows when $\widetilde{A}^{-1}$ is discrete elliptic
for all $\psi_H \in Z_H$  which can be ensured for an appropriate choice
of the finite element spaces $Z_H$ and $V_h$ where the latter has
to be defined on a sufficiently refined mesh than $Z_H$. From a practical
point of view, one additional refinement is sufficient.
Note that the above perturbed variational problem
can be written as
a mixed variational formulation to find
$(\widehat{z}_{\varrho H},p_{\widehat{z}_{\varrho H}h})
\in Z_H \times V_h$ such that
\[
  \langle p_{\widehat{z}_{\varrho H}h} , \psi_H \rangle_{L^2(\Omega)} =
  \langle \widetilde u_{\varrho h} , \psi_H \rangle_{L^2(\Omega)}, \quad
  \langle \nabla p_{\widehat{z}_{\varrho H}h} ,
  \nabla q_h \rangle_{L^2(\Omega)} =
  \langle \widehat{z}_{\varrho H} , q_h \rangle_{L^2(\Omega)}
\]
is satisfied for all $(\psi_H,q_h) \in Z_H \times V_h$.
Related error estimates rely on the use of the Strang lemma. A more
detailed numerical analysis of the approach was already given in
\cite{GanglLoescherSteinbach:2023}.
Using the fe-isomorphism $\mathbb{R}^M\ni \textbf{u}_{\varrho h}\leftrightarrow \widetilde u_{\varrho h}\in V_h$, we can reconstruct the control $\mathbb{R}^{N_H}\ni \textbf{z}_{\varrho H}\leftrightarrow \widehat z_{\varrho H}\in Z_H$ by solving 
\begin{align*}
	\begin{pmatrix}
		K_h & -\hat M_{h}^\top\\ \hat M_{h} & 0 
	\end{pmatrix}
	\begin{pmatrix}
		\textbf{p}_{h}\\ \textbf{z}_{\varrho H}
	\end{pmatrix} = 
	\begin{pmatrix}
		\textbf{0}_{h}\\ \hat M_h\textbf{u}_{\varrho h}
	\end{pmatrix}
\end{align*}
where the stiffness and mass matrices admit the entries
\begin{align}
	K_h[i,j] = \int_\Omega \nabla \varphi_j(x)\cdot \nabla \varphi_i(x)\, dx \quad \text{and}\quad \hat M_{h}[\ell,j] = \int_\Omega \varphi_j(x)\psi_\ell(x)\, dx 
\end{align}
for $i,j=1,\ldots,M$ and $\ell=1,\ldots,N_H$. To obtain stability, the coarse mesh of the control is chosen such that $h = H/4$. The reconstructed controls for the target $\overline u_{2D}$ for both a uniform refinement with $N_H=2048$ elements and an adaptive refinement with $N_H = 544$ elements are depicted in Figure \ref{fig:control-reconstruction}.  

\begin{figure}[htpb!]
	\centering
	\begin{subfigure}[b]{0.45\textwidth}
		\centering
		\includegraphics[width=\textwidth]{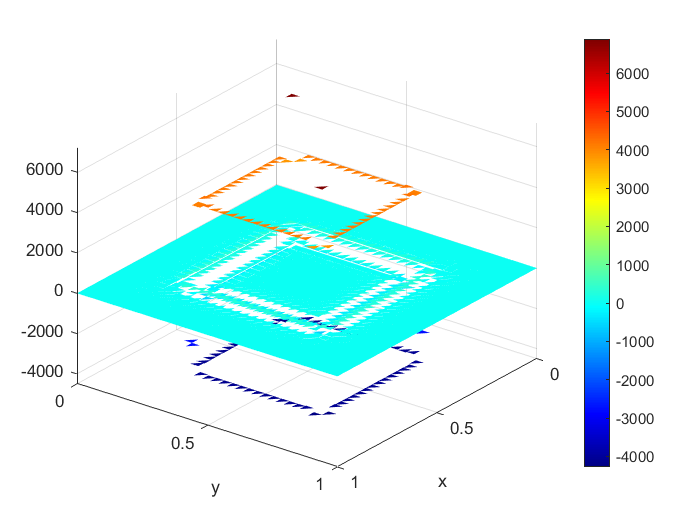}
		\caption{Uniform $N=2048$}
	\end{subfigure}
	\hfill
	\begin{subfigure}[b]{0.45\textwidth}
		\centering
		\includegraphics[width=\textwidth]{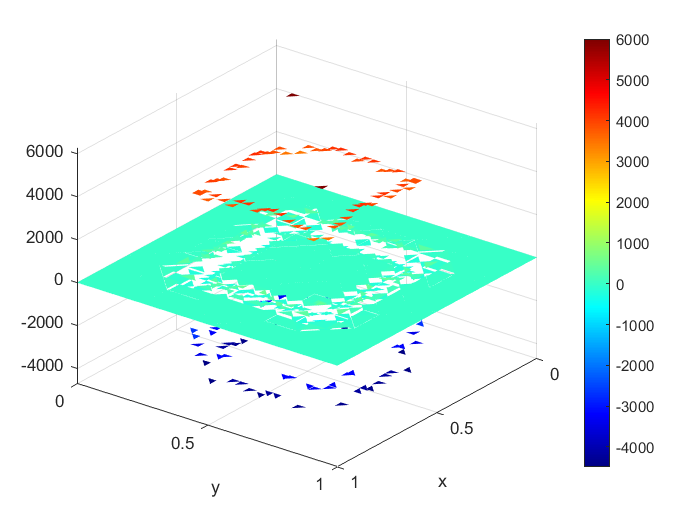}
		\caption{Adaptive $N=544$}
	\end{subfigure}
	\caption{Reconstructed controls $\hat z_{\varrho H}\in Z_H$ for the target $\overline u_{2D}$ for a uniform and an adaptive refinement. }
	\label{fig:control-reconstruction}
\end{figure} 

%
%
\subsection{Solver studies}
\label{subsec:Solver studies}
While all the numerical results 
presented in the previous subsection
were computed using Matlab with a sparse direct solver, we finally discuss
the use of preconditioned iterative solution strategies which are robust
with respect to the regularization parameter function $\varrho(x)$. Here we will
restrict our considerations to the three-dimensional case with the
target $\overline{u}_{3D}$.

We first consider the preconditioned conjugate gradient (PCG)
solver applied to the Schur complement equation
\eqref{Schur LGS} with a proper preconditioner. When performing a uniform
refinement, the diffusion coefficients (the inverse of the regularization
parameters) are constant on all elements, and we may replace $\varrho(x)$
by $h^2$ with $h$ being the global mesh size. Therefore, the Schur
complement is simplified as
$M_h+K_h K_{\varrho h}^{-1}K_h=M_h+\varrho K_h K_h^{-1}K_h=M_h+ h^2 K_h$.
Robust preconditioners for such a Schur complement have been studied in
our previous work \cite{LLSY2:LangerSteinbachYang:2022b}, and the spectral
equivalence of this Schur complement to the mass matrix was analyzed in
our recent work \cite{LLSY2:LangerLoescherSteinbachYang:2022a}. 
Since, in this case, the Schur complement
$M_h+\varrho K_h K_h^{-1}K_h=M_h+ h^2 K_h$
is spectrally equivalent to the mass matrix $M_h$,
we can use a simple diagonal approximation to the mass matrix such as 
$\textup{diag}[M_h]$ or $\textup{lump}[M_h]$ as cheap preconditioners 
for the Schur complement. In the case of an adaptive refinement, both
the local mesh refinement and the use of varying diffusion coefficients
$\varrho(x)$, which are piecewise constant, play a decisive role in
developing robust Schur complement preconditioners with respect to the mesh
size and diffusion coefficient jumps. Indeed, the following lemma states
that the lumped mass matrix $\textup{lump}[M_h]$ is spectrally equivalent
to the Schur complement $S_h = M_h+K_h K_{\varrho h}^{-1}K_h$.

\begin{lemma}
    Let us again assume that the computational domain $\Omega \subset \mathbb{R}^n$
    is decomposed into $N$ shape-regular finite elements $\tau_\ell$, $\ell=1,\ldots,N$,
    and the finite element space $V_h = \mbox{span} \{ \varphi_k \}_{k=1}^M \subset H^1_0(\Omega)$
    is spanned by continuous, piecewise linear basis functions. 
    Then the spectral equivalence inequalities
    \begin{equation}
    c_1 \, \textup{lump}[M_h] 
    \le M_h \le M_h+K_h K_{\varrho h}^{-1}K_h \le (1+c_\text{\tiny inv}^2)M_h \le
     c_2 \, \textup{lump}[M_h]
    \end{equation}
    hold, where $c_1 = 1/(n+2)$, $c_2 = 1+c_\text{\tiny inv}^2$, and $c_\text{\tiny inv}$ is nothing but the 
    constant from the local inverse inequalities 
    \begin{equation}
    \label{Eqn:LocalInverseInequalities}
     \|\nabla u_h\|_{L^2(\tau_\ell)} \le c_\text{\tiny inv}\,h_\ell^{-1}\,\|u_h\|_{L^2(\tau_\ell)}
     \quad \mbox{for all} \; u_h \in V_h,\; \ell=1,2,\ldots,N.
    \end{equation}
    %
    We note that $c_\text{\tiny inv}$ 
    is a 
    generic positive constant that can be computed from the shape-regularity parameters.
\end{lemma}
\begin{proof}
It remains to estimate $K_hK_{\varrho h}^{-1}K_h$ from above by the mass matrix $M_h$ in the spectral sense.
Using Cauchy's inequality, \eqref{pw constant regularization}, and the inverse inequalities \eqref{Eqn:LocalInverseInequalities}, 
we get the estimate
\begin{eqnarray*}
(K_hK_{\varrho h}^{-1}K_h \underline{u},\underline{u})
    &=& \sup_{\underline{q} \in \mathbb{R}^{M}} 
                            \frac{(K_h\underline{u},\underline{q})^2}{(K_{\varrho h}\underline{q},\underline{q})}\\
    &=& \sup_{q_h \in V_h}  \frac{[\int_\Omega \varrho^{1/2} \nabla u_h \cdot \varrho^{-1/2} \nabla q_h \, dx]^2}
             {\int_\Omega \varrho^{-1}  \nabla q_h \cdot  \nabla q_h \,dx}\\
    &\le& \sup_{q_h \in V_h}  \frac{\|\varrho^{1/2} \nabla u_h\|_{L^2(\Omega)}^2 \, \|\varrho^{-1/2} \nabla q_h\|_{L^2(\Omega)}^2}
    {\|\varrho^{-1/2} \nabla q_h\|_{L^2(\Omega)}^2}\\
    &=& \|\varrho^{1/2} \nabla u_h\|_{L^2(\Omega)}^2 = \int_\Omega \varrho\, \nabla u_h \cdot  \nabla u_h \, dx\\
    &=& \sum_{\ell=1}^N \int_{\tau_\ell} h_\ell^2 \, \nabla u_h \cdot  \nabla u_h \, dx
            = \sum_{\ell=1}^N h_\ell^2 \,\|\nabla u_h\|_{L^2(\tau_\ell)}^2\\
    &\le& \sum_{\ell=1}^N h_\ell^2 c_\text{\tiny inv}^2 h_\ell^{-2}\,\|u_h\|_{L^2(\tau_\ell)}^2
            = c_\text{\tiny inv}^2 (M_h \underline{u},\underline{u})
\end{eqnarray*}
for all nodal vectors $\underline{u} \in \mathbb{R}^{M}$, 
where $u_h \in V_h$ is the associated  finite element function.
The spectral equivalence inequalities 
\begin{equation}
\label{Eqn:SpectralEquivalenceMlM}
 (1/(n+2))\, \textup{lump}[M_h] \le M_h \le \textup{lump}[M_h]
\end{equation}
complete the proof. The spectral equivalence inequalities \eqref{Eqn:SpectralEquivalenceMlM} 
can easily be proven by the element matrix representation.
\end{proof}

\noindent
We note that $\textup{lump}[M_h]$ can be replaced by $\textup{diag}[M_h]$,
i.e., $\textup{diag}[M_h]$ is also spectrally equivalent to $S_h$.
The robustness with respect to the minimal and maximal
local mesh size ($h_{\min}$  and $h_{\max}$) is numerically confirmed by
the constant iteration numbers of the PCG method preconditioned by the
lumped mass matrix (Its (PCG)) on the adaptive refinements in
Table~\ref{tab:solveschurpcg}, in comparison with
the increasing conjugate gradient (CG) iteration numbers without using
the preconditioner (Its (CG)). Note that we solve the Schur complement
equation \eqref{Schur LGS}  until the relative preconditioned residual
error is reduced by a factor $10^{6}$. For the inverse operation of
$K_{\varrho h}^{-1}$ applied to a vector $\underline{v}$ within the PCG
iteration, we have used the classical Ruge--St\"{u}ben algebraic
multigrid (AMG) \cite{RugeStuben} preconditioned CG method to
solve $K_{\varrho h}\underline{w}=\underline{v}$ until the relative
preconditioned residual error reaches $10^{-12}$ in order to perform
a sufficiently accurate multiplication with the Schur complement.
Alternatively, one can here use a sparse factorization in a preprocessing step.
  
\begin{table}[h]
  \begin{tabular}{|l|rlll|ll|}
    \hline
    Level&\#Dofs&$h_{\min}$ & $h_{\max}$ & $\|\widetilde{u}_{\varrho
        h}-\overline{u}\|_{L^2(\Omega)}$ & Its (PCG) & Its (CG)\\ \hline
    $L_1$&$125$ &$2^{-2}$ &$2^{-2}$ &$3.01923$e$-1$& $7$ &$8$ \\
    $L_2$&$223$ &$2^{-3}$ &$2^{-2}$ &$2.55302$e$-1$& $18$ &$22$ \\
    $L_3$&$1,059$ &$2^{-4}$ &$2^{-2}$ &$1.79986$e$-1$& $23$ &$52$ \\
    $L_4$&$4,728$ &$2^{-5}$ &$2^{-2}$ &$1.26353$e$-1$& $26$ &$127$ \\
    $L_5$&$18,827$ &$2^{-6}$ &$2^{-2}$ &$8.84306$e$-2$& $25$ &$317$ \\
    $L_6$&$75,603$ &$2^{-7}$ &$2^{-2}$ &$6.21010$e$-2$& $24$ &$829$ \\
    $L_7$&$303,782$ &$2^{-8}$ &$2^{-2}$ &$4.37439$e$-2$& $23$ &$2,103$ \\
    $L_8$&$1,218,846$ &$2^{-9}$ &$2^{-2}$ &$3.08691$e$-2$& $23$ &$5,219$ \\
    $L_9$&$4,884,317$ &$2^{-10}$ &$2^{-2}$ &$2.18069$e$-2$& $19$ &$>10,000$ \\
    $L_{10}$&$19,553,202$ &$2^{-11}$ &$2^{-2}$ &$1.54097$e$-2$& $17$ &$>20,000$ \\
    $L_{11}$&$78,277,988$ &$2^{-12}$ &$2^{-2}$ &$1.08918$e$-2$& $15$ &$>40,000$ \\
    \hline
  \end{tabular}
  \caption{Comparison of the PCG (Its (PCG)) and CG iterations (Its (CG))
    for the Schur complement equation \eqref{Schur LGS} on the adaptive
    refinements.}  
  \label{tab:solveschurpcg}
\end{table}

\noindent
Furthermore, we provide some numerical results concerning robust solvers
for the coupled optimality system \eqref{coupled LGS}:
\begin{equation}\label{eq:unsym_system}
  \begin{bmatrix}
    K_{\varrho h} & K_h\\
    -K_h & M_h
  \end{bmatrix}
  \begin{bmatrix}
    \underline{p}\\
    \underline{u}
  \end{bmatrix}
  =
  \begin{bmatrix}
    \underline{0}\\
    \underline{f}
  \end{bmatrix}.
\end{equation}
Since the system matrix is non-symmetric and positive definite, we
apply the GMRES method with the following proposed block diagonal
preconditioner:
\begin{equation*}
  {\mathcal P}_h =
  \begin{bmatrix}
    \widehat{K}_{\varrho h} & 0\\
    0 & \textup{lump}[M_h]
  \end{bmatrix}.
\end{equation*}
The number of GMRES iterations (Its) using such a preconditioner are given
in Table~\ref{tab:solvecompgmres}. We solve the system until the relative
preconditioned residual error is reduced by a factor $10^{6}$. In the
preconditioner ${\mathcal P}_h$, we have utilized the Ruge--St\"{u}ben AMG preconditioner $\widehat{K}_{\varrho h}$ for $K_{\varrho h}$,
whereas the lumped mass matrix $\textup{lump}[M_h]$ has been used as
preconditioner for the Schur complement. We observe almost constant
iteration numbers of the GMRES method preconditioned by ${\mathcal P}_h$
on the uniform refinement as well as on the adaptive refinements as given
in Table \ref{tab:solvecompgmres}. We only see slightly higher iteration
numbers on the adaptive meshes than on the uniform ones.

\begin{table}[h]
  \begin{tabular}{|l|rlr|rlr|}
    \hline
    \multirow{2}{*}{Level}&\multicolumn{3}{c|}{Adaptive} &
    \multicolumn{3}{c|}{Uniform} \\  \cline{2-7}
      &\#Dofs&$\|\widetilde{u}_{\varrho h}-\overline{u}\|_{L^2(\Omega)}$ &Its& \#Dofs & $\|\widetilde{u}_{\varrho h}-\overline{u}\|_{L^2(\Omega)}$ &Its\\
      \hline
      $L_1$&$250$ &$3.01923$e$-1$&$14$ &$250$&$3.01923$e$-1$&$14$ \\
      $L_2$&$446$&$2.55302$e$-1$&$33$ &$1,458$&$2.26341$e$-1$&$28$ \\
      $L_3$&$2,118$&$1.79985$e$-1$&$40$ &$9,826$&$1.61850$e$-1$&$36$ \\
      $L_4$&$9,456$&$1.26356$e$-1$&$46$ &$71,874$&$1.14659$e$-1$&$36$ \\
      $L_5$&$37,656$&$8.84293$e$-2$&$48$ &$549,250$&$8.10582$e$-2$&$34$ \\
      $L_6$&$151,212$&$6.20997$e$-2$&$47$ &$4,293,378$&$5.72923$e$-2$&$32$ \\
      $L_7$&$607,586$&$4.37450$e$-2$&$46$ &$33,949,186$&$4.04944$e$-2$&$30$ \\
      $L_8$&$2,437,880$&$3.08701$e$-2$&$44$ &$270,011,394$&$2.86176$e$-2$& $28$ \\
      $L_9$&$9,769,496$&$2.18075$e$-2$&$44$ &&& \\
      $L_{10}$&$39,108,934$&$1.54100$e$-2$&$42$ &&&\\
      $L_{11}$&$156,568,020$&$1.08922$e$-2$&  $42$&&&\\ 
      \hline
    \end{tabular}
    \caption{Comparison of the preconditioned GMRES solver on both
      the adaptive and uniform refinements.} 
    \label{tab:solvecompgmres}
\end{table}

\noindent
On the other hand, we may reformulate the coupled system
\eqref{eq:unsym_system} in the following equivalent form 
\begin{equation}\label{eq:sym_system}
  \begin{bmatrix}
    K_{\varrho h} & K_h\\
    K_h & -M_h
  \end{bmatrix}
  \begin{bmatrix}
    \underline{p}\\
    \underline{u}
  \end{bmatrix}
  =
  \begin{bmatrix}
    \underline{0}\\
    -\underline{f}
  \end{bmatrix}.
\end{equation}
When applying the Bramble--Pasciak transformation
\begin{equation*}
  \mathcal{T}_h=
  \begin{bmatrix}
    K_{\varrho h} C_{h}^{-1}-I_h                 &  0 \\
    K_h C_{h}^{-1}     &  - I_h\\
  \end{bmatrix}
\end{equation*}
to the symmetric but indefinite system \eqref{eq:sym_system}, this leads to
the equivalent system
\begin{equation*}\label{eqn:spdBPsystem}
  \mathcal{K}_h 
  \begin{bmatrix}
    \underline{p}\\
    \underline{u}
  \end{bmatrix} 
  = 
  \begin{bmatrix}
    \underline{0}\\
    \underline{f}
  \end{bmatrix}
  \equiv
  \mathcal{T}_h
  \begin{bmatrix}
    \underline{0}\\
    -\underline{f}
  \end{bmatrix},
\end{equation*}
with the symmetric and positive definite system matrix
\begin{equation*}
  \begin{aligned}
    \mathcal{K}_h&=
    \begin{bmatrix}
      K_{\varrho h} C_{h}^{-1}-I_h                 &  0 \\
      K_h C_{h}^{-1}     &  - I_h\\
    \end{bmatrix}
    \begin{bmatrix}
      K_{\varrho h} & K_h\\
      K_h & -M_h
    \end{bmatrix}\\
    &=
    \begin{bmatrix}
      (K_{\varrho h}- C_{h})C_{h}^{-1}K_{\varrho h}                 &
      (K_{\varrho h}-C_h) C_h^{-1} K_h \\
      K_h C_h^{-1}(K_{\varrho h}-C_h)&  K_hC_h^{-1}K_h+M_h\\
    \end{bmatrix}.
  \end{aligned}
\end{equation*}
Here, $I_h$ denotes the identity, and $C_h$ is some symmetric and positive
definite (spd) matrix that is assumed to be spectrally equivalent to the
matrix $K_{\varrho h}$, and less than $K_{\varrho h}$, i.e., $C_h<K_{\varrho h}$. 
In particular, we can again take the classical spd
Ruge--St\"{u}ben AMG preconditioner 
$\widehat{K}_{\varrho h} = \delta K_{\varrho h}(I_h - N_{\varrho h}^j)^{-1}$
with a proper scaling $\delta > 0$ as $C_h$,
where 
$N_{\varrho h}$
denotes the AMG iteration matrix,
and $j$ the number of AMG cycles; see, e.g.,
\cite{LLSY2:TrottenbergOosterleeSchueller:2001Monograph}.
In our numerical experiments, we have chosen $\delta=0.5$, 
and we use the AMG V-cycle with
$2$ forward Gauss-Seidel pre-smoothing steps and $2$ backward Gauss-Seidel post-smoothing steps at all levels,
where $j$ is equal to $8$ and $12$ for uniform and adaptive refinements, respectively.
Now using $C_h = \widehat{K}_{\varrho h}$ and the lumped mass matrix
$\textup{lump}[M_h]$ as preconditioner for the Schur complement $M_h + K_h K_{\varrho h}^{-1} K_h$,
we arrive at the following (inexact) BP preconditioner: 
\begin{equation*}
    {\mathcal P}_h=
    \begin{bmatrix}
     K_{\varrho h} - C_h & 0\\
      0 & \textup{lump}[M_h]
    \end{bmatrix}.
\end{equation*}
Details on the BP-CG can be found in the original paper
\cite{LLSY2:BramblePasciak:1988a}; see also \cite{LLSY2:Zulehner:2001a}
for improved convergence rate estimates.
The number of BP-CG iterations (Its) using the  preconditioner
${\mathcal P}_h$ are provided in Table~\ref{tab:solvecomppbcg}. The solver
stops the iteration when the relative preconditioned residual error is
reduced by a factor $10^{6}$. From the constant PB-CG iteration numbers
on both the uniform and adaptive mesh refinements, we observe the
robustness of the proposed preconditioner for the coupled optimality system. 

\begin{table}[h]
    \begin{tabular}{|l|rlr|rlr|}
      \hline
      \multirow{2}{*}{Level}&\multicolumn{3}{c|}{Adaptive} &
      \multicolumn{3}{c|}{Uniform} \\  \cline{2-7}
      &\#Dofs&$\|\widetilde{u}_{\varrho h}-\overline{u}\|_{L^2(\Omega)}$ &Its& \#Dofs & $\|\widetilde{u}_{\varrho h}-\overline{u}\|_{L^2(\Omega)}$ &Its\\
      \hline
      $L_1$&$250$ &$3.01923$e$-1$&$13$ &$250$&$3.01923$e$-1$ &$13$ \\
      $L_2$&$446$&$2.55302$e$-1$&$32$ &$1,458$&$2.26348$e$-1$&$25$ \\
      $L_3$&$2,118$&$1.79986$e$-1$&$34$ &$9,826$&$1.61863$e$-1$&$30$ \\
      $L_4$&$9,456$&$1.26353$e$-1$&$38$ &$71,874$&$1.14657$e$-1$&$31$ \\
      $L_5$&$37,654$&$8.84307$e$-2$&$39$ &$549,250$&$8.10658$e$-2$&$25$ \\
      $L_6$&$151,206$&$6.21010$e$-2$&$38$ &$4,293,378$&$5.73048$e$-2$&$24$\\
      $L_7$&$607,558$&$4.37438$e$-2$&$36$ &$33,949,186$&$4.05125$e$-2$&$23$ \\
      $L_8$&$2,437,674$&$3.08690$e$-2$&$35$ &$270,011,394$&$2.86430$e$-2$& $23$\\
      $L_9$&$9,768,526$&$2.18067$e$-2$&$33$ &&& \\
      $L_{10}$&$39,105,552$&$1.54096$e$-2$&$34$&&&\\
      $L_{11}$&$156,550,890$&$1.089170$$-2$&$30$&&&\\ 
      \hline
    \end{tabular}
\caption{Comparison of the preconditioned PB-CG solver on both the adaptive and
  uniform refinements.} 
\label{tab:solvecomppbcg}
\end{table}
%
%

Now we can use the preconditioned PB-CG solver in 
a nested iteration process that interpolates the iterative approximation 
from the coarser mesh in order to obtain a good initial guess; see, e.g., \cite{LLSY2:Hackbusch:2016Monograph}. 
Table~\ref{tab:solvecomppbcgnested} shows 
that we can obtain approximate solutions $\widetilde{u}_{\varrho h}$, 
which differ from the desired state $\overline{u}$ in the order of the discretization error,
with considerable less iterations that are pretty constant across the levels;
cf.~also with Table~\ref{tab:solvecomppbcg}.

\begin{table}[h]
    \begin{tabular}{|l|rlr|rlr|}
      \hline
      \multirow{2}{*}{Level}&\multicolumn{3}{c|}{Adaptive} &
      \multicolumn{3}{c|}{Uniform} \\  \cline{2-7}
      &\#Dofs&$\|\widetilde{u}_{\varrho h}-\overline{u}\|_{L^2(\Omega)}$ &Its& \#Dofs & $\|\widetilde{u}_{\varrho h}-\overline{u}\|_{L^2(\Omega)}$ &Its\\
      \hline
      $L_1$&$250$ &$3.01923$e$-1$&$13$ &$250$&$3.01923$e$-1$ &$13$ \\
      $L_2$&$446$&$2.55278$e$-1$&$12$ &$1,458$&$2.26412$e$-1$&$10$ \\
      $L_3$&$2,118$&$1.80027$e$-1$&$13$ &$9,826$&$1.61959$e$-1$&$12$ \\
      $L_4$&$9,458$&$1.26535$e$-1$&$16$ &$71,874$&$1.14724$e$-1$&$12$ \\
      $L_5$&$37,696$&$8.84904$e$-2$&$16$ &$549,250$&$8.11476$e$-2$&$11$ \\
      $L_6$&$151,220$&$6.21325$e$-2$&$16$ &$4,293,378$&$5.73597$e$-2$&$11$\\
      $L_7$&$607,170$&$4.37623$e$-2$&$16$ &$33,949,186$&$4.05516$e$-2$&$11$\\
      $L_8$&$2,435,524$&$3.08822$e$-2$&$16$ &$270,011,394$&$2.86695$e$-2$& $11$\\
      $L_9$&$9,758,590$&$2.18163$e$-2$&$16$ &&& \\
      $L_{10}$&$39,065,548$&$1.54165$e$-2$&$16$&&&\\
      $L_{11}$&$156,371,118$&$1.089660$$-2$&$16$&&&\\ 
      \hline
    \end{tabular}
\caption{Comparison of the preconditioned PB-CG solver on both the adaptive and
  uniform refinements using nested iteration, where the stopping criterio for
  the nested iteration is: the relative preconditioned residual error of the
  full system (including both state and adjoint states) is smaller than
  $\alpha[n_l/n_{l-1}]^{-\frac{\beta}{3}}$, $l=2,3,...$, $\beta=0.5$
  (uniform), $\beta=0.75$ (adaptive), and $\alpha=0.025$. 
  } 
\label{tab:solvecomppbcgnested}
\end{table}

%
%
\section{Conclusion and outlook}
\label{sec:ConclusionAndOutlook}
We have studied finite element discretizations of the reduced optimality
system for the standard distributed space-tracking elliptic optimal control
problem, but using a new variable energy regularization technique. It has
been shown that the choice of the local mesh-size squared as local
regularization parameter $\varrho(x)$ leads to optimal rates of convergence
of the computed finite element state $\widetilde{u}_{\varrho h}$ to the
prescribed target $\overline{u}$ in the $L^2$ norm. In particular,
this approach allows us to adapt the local regularization parameter 
to the local mesh-size when using an adaptive mesh refinement,
where the adaptivity is driven by the localization of the $L^2$ norm of
the error between $\overline{u}$ and $\widetilde{u}_{\varrho h}$ 
as computable local error indicator. Numerical studies made for
discontinuous targets in one, two and three space dimensions illustrate
that these simple adaptive schemes show a significantly better
performance than the uniform refinement.
The control can easily be recovered from the computed state 
in a postprocessing procedure. We have also proposed 
iterative solvers for the finite element 
equations corresponding to the reduced optimality system.
The numerical studies have shown that these solvers are robust and efficient 
at the same time. This behavior is based on the fact 
that the mass matrix $M_h$, and, therefore, also the lumped
mass matrix $\text{lump}[M_h]$ are spectrally equivalent to the Schur
complement $K_h K_{\varrho h}^{-1} K_h + M_h$, and can be used as robust
preconditioners in the preconditioned BP-CG, MINRES, or GMRES solvers. 
Obviously, the classical plain Ruge--St\"{u}ben AMG
preconditioner is doing this job for $K_{\varrho h}$.
Here we may develop more efficient and robust preconditioners 
that are especially adapted to the
diffusion coefficient $\varrho(x)$ that changes from
element to element according to the mesh-sizes. In a possibly adaptive,
multilevel setting, the nested iteration technique can be used to 
compute state approximations 
which differ from the desired state in the order of the discretization error
in asymptotically optimal complexity. The computation of the control can 
be integrated in the nested iteration process that can then be stopped 
if the cost of the computed control exceeds some threshold or the
required approximation of the desired state is reached.
Moreover, the parallelization of such iterative solution
strategies should be  implemented in order to solve really large-scale systems
in three space dimensions.

Finally, the variable energy regularization and the robust solvers for
the corresponding linear system of algebraic equations can be extended
to optimal control problems with state or control constraints; see
\cite{GanglLoescherSteinbach:2023} 
for the case of a constant regularization parameter.
Furthermore, the results can be generalized to other state equations
like elasticity, Maxwell, and Stokes equations, but also to time-dependent 
problems such as parabolic and hyperbolic initial boundary value problems
\cite{UL:LangerSteinbachYang:2022a, LoescherSteinbach:2022}.

%
%

\section*{Declarations}
{\bf Conflict of interest:} The authors declared that they have
no conflict of interest. \\[1mm]
{\bf Data availability:}
Data will be made available on request.

\section*{Acknowledgments}
We would like to thank the computing resource support of the supercomputer
MACH--2\footnote{https://www3.risc.jku.at/projects/mach2/} from Johannes
Kepler Universit\"{a}t Linz and of the high performance computing cluster
Radon1\footnote{https://www.oeaw.ac.at/ricam/hpc} from Johann Radon Institute
for Computational and Applied Mathematics. 
Further, the financial support for the fourth author by the
Austrian Federal Ministry for Digital and Economic Affairs, the National
Foundation for Research, Technology and Development and the Christian
Doppler Research Association is gratefully acknowledged.
We finally thank B.~Kaltenbacher for pointing out the references
on the use of variable regularization techniques in imaging.


\bibliography{paperadaptivR1}
\bibliographystyle{abbrv} 
  

\end{document}